\documentclass[english, 10pt,reqno]{amsart}
\usepackage{geometry}  
\usepackage{amssymb,amsmath,amsthm,amsfonts,color}
\usepackage{mathrsfs,dsfont, comment,mathscinet}
\usepackage{graphicx}
\usepackage{epstopdf}
\usepackage{mathtools}
\usepackage{babel}
\usepackage{enumerate,esint}

\usepackage{indentfirst}
\usepackage{bm}
\usepackage{picinpar}
\usepackage{caption}
\usepackage{subcaption}
\usepackage{soul}

\usepackage[colorlinks=true, pdfstartview=FitV, linkcolor=blue, citecolor=blue, urlcolor=blue]{hyperref}

\usepackage[toc,page]{appendix}

\usepackage{etoolbox}
\usepackage{authblk}
\usepackage{amsaddr}

\usepackage[running]{lineno}

\newcommand*\patchAmsMathEnvironmentForLineno[1]{%
  \expandafter\let\csname old#1\expandafter\endcsname\csname #1\endcsname
  \expandafter\let\csname oldend#1\expandafter\endcsname\csname end#1\endcsname
  \renewenvironment{#1}%
     {\linenomath\csname old#1\endcsname}%
     {\csname oldend#1\endcsname\endlinenomath}}%
\newcommand*\patchBothAmsMathEnvironmentsForLineno[1]{%
  \patchAmsMathEnvironmentForLineno{#1}%
  \patchAmsMathEnvironmentForLineno{#1*}}%
\AtBeginDocument{%
\patchBothAmsMathEnvironmentsForLineno{equation}%
\patchBothAmsMathEnvironmentsForLineno{align}%
\patchBothAmsMathEnvironmentsForLineno{flalign}%
\patchBothAmsMathEnvironmentsForLineno{alignat}%
\patchBothAmsMathEnvironmentsForLineno{gather}%
\patchBothAmsMathEnvironmentsForLineno{multline}%
}
\allowdisplaybreaks 

\mathtoolsset{showonlyrefs} 

\makeatletter

\usepackage{fancyhdr}
 
\makeatletter
\patchcmd{\@maketitle}
  {\ifx\@empty\@dedicatory}
  {\ifx\@empty\@date \else {\vskip3ex \centering\footnotesize\@date\par\vskip1ex}\fi
   \ifx\@empty\@dedicatory}
  {}{}
\patchcmd{\@maketitle}
  {\ifx\@empty\@date\else \@footnotetext{\@setdate}\fi}
  {}{}{}
\makeatother

\begingroup
\newtheorem{theorem}{Theorem}[section]
\newtheorem{lemma}[theorem]{Lemma}
\newtheorem{proposition}[theorem]{Proposition}

\endgroup

\theoremstyle{definition}
\begingroup
\newtheorem{define}[theorem]{Definition}
\newtheorem{assumption}[theorem]{Assumption}
\newtheorem{remark}[theorem]{Remark}

\endgroup

\newcommand\ba[1]{\begin{align}\label{#1}}
\newcommand\ea{\end{align}}
\newcommand\bas{\begin{align*}}
\newcommand\eas{\end{align*}}

\newcommand\ee{\end{equation}}
\newcommand\be{\begin{equation}}
\newcommand\ees{\end{equation*}}
\newcommand\bes{\begin{equation*}}

\mathchardef\emptyset="001F

\newcommand{\e}{\varepsilon}

\newcommand{\R}{{\mathbb R}}

 \newcommand{\8}{\infty}

\newcommand{\I}{{\mathcal I}}

\newcommand{\rn}{{{\R}^N}}

\newcommand{\wto}{\rightharpoonup}
\newcommand{\wtos}{\mathrel{\mathop{\rightharpoonup}\limits^*}}

\newcommand{\N}{{\mathbb{N}}}

\newcommand{\B}{{\mathscr B}}
\newcommand{\FE}{{\mathcal E}}

\newcommand\norm[1]{\left\|#1\right\|}

\newcommand{\abs}[1]{\left\lvert#1\right\rvert} 
\newcommand{\fsp}[1]{\left(#1\right)} 
\newcommand{\fmp}[1]{\left[#1\right]}
\newcommand{\flp}[1]{\left\{#1\right\}}

\newcommand{\vp}{\varphi}

\newcommand{\limn}{\lim_{n\rightarrow\infty}}

\newcommand{\divg}{{\operatorname{div}}}

\newcommand{\seqn}[1]{\left\{#1\right\}_{n=1}^\infty}
\newcommand{\seqk}[1]{\left\{#1\right\}_{k=1}^\infty}  

\newcommand{\seqe}[1]{\left\{#1\right\}_{\e>0}}

\newcommand{\liminfn}{{\liminf_{n\to\infty}}}
\newcommand{\ir}{{\lfloor r\rfloor}}

\newcommand\M{\mathbb M}

\definecolor{CMUred}{RGB}{153,0,0}
\definecolor{CMUgreen}{RGB}{0,135,81}
\definecolor{CMUblue}{RGB}{0,51,127}
\definecolor{Pblue}{RGB}{87,158,208}

\newcommand{\argmin}{{\operatorname{arg\,min}}}

\newcommand\T{\mathbb{T}}

\newcommand\mb{{\mathcal{M}_b}}

\newcommand{\FT}{{\mathcal T}}

\newcommand{\TGV}{{\operatorname{TGV}}}

\newcommand{\RVL}{{\operatorname{RVL}}}

\def\argmin{\mathop{\rm arg\, min}}

\numberwithin{equation}{section}

\newcommand{\normmm}[1]{{\left\vert\kern-0.25ex\left\vert\kern-0.25ex\left\vert #1 
    \right\vert\kern-0.25ex\right\vert\kern-0.25ex\right\vert}}

\title{Real order (an)-isotropic total variation in image processing - Part II: Learning of optimal structures}

\author[P. Liu] {Pan Liu}
 \address[Pan Liu]{Centre of Mathematical Imaging and Healthcare,\\ 
Department of Pure Mathematics and Mathematical Statistics, \\
 University of Cambridge,\\
  Wilberforce Road, Cambridge CB3 0WA, UK}
 \email[P. Liu] {panliu.0923@maths.cam.ac.uk}

 \author[X.Y. Lu] {Xin Yang Lu}
 \address[Xin Yang Lu]{Department of Mathematical Sciences, Lakehead University,
 955 Oliver Road, Thunder Bay, ON, Canada\\
 AND\\
 Department of Mathematics and Statistics
Burnside Hall, McGill University\\805 Sherbrooke Street West, Montreal, QC, Canada }
 \email[X. Lu] {xlu8@lakeheadu.ca}


\date{\today}                                           

\begin{document}


\begin{abstract}
This article is the second work in our series of papers dedicated to image processing models
based on the fractional order total variation $TV^r$. In our first work \cite{liulud2019Image}, we studied key analytic properties of these semi-norms.
 Here we focus on the more applied aspects
of such models: first, in order to obtain a better reconstructed image, we propose several extensions of the fractional order total variation. Such generalizations, collectively denoted by $\RVL$, will be modular, i.e. the parameters therein are mutually independent, and can be fine tuned to the particular task. Then, we will study the bilevel training schemes based on $\RVL$, and show that such schemes are well defined, i.e. they admit minimizers. Finally, we will provide some numerical examples, showing that training schemes based on $\RVL$ are effectively better than those based on classical regularizer $TV$ and $TGV^2$.
\end{abstract}

\maketitle
\tableofcontents

\thispagestyle{empty}

\section{Introduction}\label{sec:intro}
This article is the second of our series of works on image processing models
based on the real order total variation. In the previous paper \cite{liulud2019Image}, 
we introduced the semi-norm
\be\label{mumfordshahori233}
TV^r(u):=\sup\flp{\int_Qu\, [\divg^s\divg^\ir \vp] \,dx:\text{ for }\vp\in C_c^\infty(Q;\M^{N\times (N^{k})})\text{ and }\abs{\vp}\leq 1}.
\ee
Here $r=\ir+s\in\R^+$, and the fractional order derivative $\divg^s$ is realized by the \emph{Riemann-Liouville} fractional derivative (see \eqref{R_L_frac_1d_left}). 
We also write $r=\ir+s$ for $\ir\in \N$ and $s\in[0,1)$, and we denote by
\be
BV^r(Q):=\flp{u\in L^1(Q):\,\, TV^r(u)<+\infty}
\ee
the space of functions with bounded $r$-order total variation.\\\\
In this article we focus on the applications of such real order total variation to imaging processing problems, as well as a bilevel training scheme which determines the optimal parameters used in 
the underlying variational model. As we have observed in our previous work \cite{liulud2019Image}, 
the definition of fractional order \emph{Riemann-Liouville} derivative requires rather strict boundary conditions on $u$, to prevent singularities from arising at the boundary. 
Especially in numerical realization, inaccurate boundary conditions could generate oscillations around the boundaries. To overcome this issue, in \cite{zhang2015total} 
(also see  Section \ref{the_greatest_saver}), the authors introduced a modified method, tailored for image applications, to reduce the non-zero \emph{Dirichlet} boundary 
condition to a zero \emph{Dirichlet} boundary condition. In this spirit, in this article, we consider only functions $u$ such that 
\begin{itemize}
\item
$u$ has zero boundary conditions (depending on the order $r$, see Section \ref{the_greatest_saver} for details).
\end{itemize}
The aim of this article is twofold. We shall construct a new family of regularizers, based on \eqref{mumfordshahori233}, in a modular way, and then coupling it with the bilevel training scheme.
In this way, better imaging processing results can be achieved.\\\\
The \emph{bilevel training scheme}, arising in machine learning, is a semi-supervised training scheme that optimally adapts itself to the given ``perfect data" 
(see \cite{domke2012generic, domke2013learning, tappen2007utilizing,tappen2007learning}). To apply such training scheme to image processing problem, we assume that we have a pair 
of images $u_\eta\in L^2(Q)$ and $u_c\in BV(Q)$, representing the corrupted image and the corresponding clean one.
Then, a simple implementation of such bilevel training scheme with the standard $TV$ regularizer, which we call scheme $\mathcal B$, is
\begin{flalign}
\text{Level 1. }&\,\,\,\,\,\,\,\,\,\,\,\,\,\,\,\,\,\,\,\alpha_\T\in\mathbb A[\T]:=\argmin\flp{\norm{u_\alpha-u_c}_{L^2(Q)}^2:\,\,\alpha\in\T:=[0,T]}\tag{$\mathcal B$-L1},\label{intro_B_train_level1}&\\
\text{Level 2. }&\,\,\,\,\,\,\,\,\,\,\,\,\,\,\,\,\,\,\,u_{\alpha}:=\argmin\flp{\norm{u-u_\eta}_{L^2(Q)}^2+\alpha TV(u)+\kappa H^2(u):\right.\\
&\,\,\,\,\,\,\,\,\,\,\,\,\,\,\,\,\,\,\,\,\,\,\,\,\,\,\,\,\,\,\,\,\,\,\,\,\,\,\,\,\,\,\,\,\,\,\,\,\,\,\,\,\,\,\,\,\,\,\,\,\,\,\,\,\,\,\,\,\,\,\,\,\,\,\,\,\,\,\,\,\,\,\,\,\,\,\,\,\,\,\,\,\,\,\,\,\,\,\,\,\,\,\,\,\,\,\,\,\,\,\,\,\,\,\,\,\,\,\,\,\,\,\,\,\,\,\,\,\,\,\,\,\,\left.u\in BV(Q)\cap H_0^2(Q)},\tag{$\mathcal B$-L2}\label{intro_B_train_level2}&
\end{flalign}
where: 
\begin{itemize}
\item
$\T:=[0,T]$ is the {\em training ground}, 
\item
 $T\in\R^+$ is a parameter chosen by the user, usually called the \emph{box-constraint} (see, e.g. \cite{bergounioux1998optimal, de2013image}) of the training parameter, 
 \item
 $\kappa H^2(u)$ is the \emph{Huber}-regularization (see Section \ref{the_greatest_saver}), and $\kappa>0$ is a small constant,
 \item
 the space $H_0^2(Q)$ is to enforce the corresponding zero boundary conditions,
 \item
 the \emph{minimum assessment value} (MAV) is defined to be the value
 \be\label{mav_value}
 \inf\flp{\norm{u_\alpha-u_c}_{L^2(Q)}^2:\,\,\alpha\in\T:=[0,T]},
 \ee
 i.e., the minimum distance between the clean image $u_c$ to the optimal reconstructed image $u_{\alpha_\T}$ provided by the current training scheme.
\end{itemize}
In this way, the training scheme $\mathcal B$, \eqref{intro_B_train_level1}-\eqref{intro_B_train_level2}, provides the optimal intensity parameter $\alpha_\T$ for the given training set $u_\eta$ and $u_c$.\\\\
We note that the upper level problem \eqref{intro_B_train_level1}  optimizes the reconstructed image $u_\alpha$ by adjusting only the value of 
the intensity parameter $\alpha$. Thus, to improve the training result, we can replace the $TV$ semi-norm in \eqref{intro_B_train_level2} with 
the more general real order $TV^{s}$ semi-norm, with $s\in[0,1]$, and hence expand the training options. To this aim, we introduce the following scheme, denoted by
$\mathcal B'$: 
\begin{flalign}
\text{Level 1. }&\,\,\,\,\,\,\,\,\,\,\,\,\,\,\,\,(\alpha_{\T'},s_{\T'})\in\mathbb A[\T']:=\argmin\flp{\norm{u_{\alpha,s}-u_c}_{L^2(Q)}^2:\right.\\
&\,\,\,\,\,\,\,\,\,\,\,\,\,\,\,\,\,\,\,\,\,\,\,\,\,\,\,\,\,\,\,\,\,\,\,\,\,\,\,\,\,\,\,\,\,\,\,\,\,\,\,\,\,\,\,\,\,\,\,\,\,\,\,\,\,\,\,\,\,\,\,\,\,\,\,\,\,\,\,\,\,\,\,\,\,\,\,\,\,\,\,\,\,\,\,\,\,\,\,\,\,\,\,\,\,\,\,\,\,\,\,\,\,\,\left.(\alpha,s)\in\T':=[0,T]\times[0,1]}\tag{$\mathcal B'$-L1}\label{intro_B_p_train_level1}&\\
\text{Level 2. }&\,\,\,\,\,\,\,\,\,\,\,\,\,\,\,\,u_{\alpha,s}:=\argmin\flp{\norm{u-u_\eta}_{L^2(Q)}^2+\alpha TV^s(u)+\kappa H^{2}(u):\right.\\
&\,\,\,\,\,\,\,\,\,\,\,\,\,\,\,\,\,\,\,\,\,\,\,\,\,\,\,\,\,\,\,\,\,\,\,\,\,\,\,\,\,\,\,\,\,\,\,\,\,\,\,\,\,\,\,\,\,\,\,\,\,\,\,\,\,\,\,\,\,\,\,\,\,\,\,\,\,\,\,\,\,\,\,\,\,\,\,\,\,\,\,\,\,\,\,\,\,\,\,\,\,\,\,\,\,\,\,\,\,\,\,\,\,\,\left.u\in BV(Q)\cap H_0^2(Q)},\tag{$\mathcal B'$-L2}\label{intro_B_p_train_level2}&
\end{flalign}
That is, in scheme $\mathcal B'$, \eqref{intro_B_p_train_level1}-\eqref{intro_B_p_train_level2}, the upper level problem \eqref{intro_B_p_train_level1} is able to optimize the reconstructed image $u_{\alpha,s}$ by adjusting simultaneously 
the intensity parameter $\alpha$ and the order $s\in[0,1]$. Thus, the new scheme $\mathcal B'$ provides improved results compared to scheme $\mathcal B$, as it has more training options.\\\\
Following this spirit, we could continue to improve the training result if we are able to further generalize the regularizer $TV^s$. This is one of the main topics of this article. \\\\
Succintly, we generalize the total variation $TV$ term in two steps. For simplicity, at each step, we focus on generalzing only one parameter.
\begin{enumerate}[Step 1.]
\item
The primal extension on total variation (Section \ref{sec_extension_primal})
\begin{enumerate}[1.]
\item
(Section \ref{sec_bilevel_TVr}) Generalize $TV$ to the family of real order total variations
\be\label{intro_TVrr}
\flp{TV^r:\,\,r\in[0,T]}.
\ee
\item
(Section \ref{sec_bilevel_TVp}) Extend \label{intro_TVr} to
\be\label{intro_TVp}
\flp{TV^r_{\ell^p}:\,\,r\in[0,T],\,\,p\in[1,+\infty]},
\ee
i.e., the family of $\ell^p$-(an)-isotropic total variations, where $\ell^p$ denotes the underlying Euclidean norm used to define the total variation seminorm.
\item
Combine \eqref{intro_TVrr} and \eqref{intro_TVp} to obtain collection 
\be\label{intro_TV_pr}
\flp{TV^r_{\ell^p}:\,\,r\in[0,T],\,\,p\in[1,+\infty]}.
\ee
\end{enumerate}
\item
The \emph{infimal convolution} extension (Section \ref{sec_extension_infimal})
\begin{enumerate}[1.]
\item
(Section \ref{sec_tgv_r}) Extending $TV$ to collection 
\be\label{intro_TGV}
\flp{\TGV^{r,\kappa}:\,\,r\in[1,T]},
\ee
 the collection of real order total generalized variation, where $\kappa\in(0,1)$ is the embedded Huber-regularization.
\item
(Section \ref{sec_tgv_rp}) Extending the family \eqref{intro_TGV} to  
\be\label{intro_TGV_pr}
\flp{\TGV_{p}^{r,\kappa}:\,\,r\in[1,T],\,\,p\in[1,+\infty]}
\ee
\end{enumerate}
\item
(Section \ref{sec_RVL_final}) The final collection 
\be
\flp{\RVL^{r,\kappa}_{\alpha,p}:\,\,(\alpha,r,p)\in \T}
\ee
 is obtained by combining the families in \eqref{intro_TV_pr} and \eqref{intro_TGV_pr}, where 
\be\label{big_training_set}
\T:=[0,T]^{T+1}\times[1,T]^2\times [1,+\infty]^{T+1}.
\ee
That is, we have here $r=(r_1,r_2)\in\R^2_+$.
\end{enumerate}
The collection of regularizers $\RVL^{r,\kappa}_{\alpha,p}$ provides a significant generalization of the original total variation semi-norm $TV$, 
and also provides an unified approach to widely used regularizers such as $TV$ and $TGV$. Therefore, the new training scheme $\mathcal T$, equipped with such regularizer, 
is given by
\begin{flalign}
\text{Level 1. }&\,\,\,\,\,\,\,\,\,\,\,\,(\alpha_{\T}, r_{\T},p_\T)\in\mathbb A[\T]:=\argmin\flp{\norm{u_{\alpha,r,p}-u_c}_{L^2(Q)}^2:\,\,(\alpha,r,p)\in\mathbb T}\tag{$\mathcal T$-L1}\label{S_scheme_L1_D0_intro}&\\
\text{Level 2. }&\,\,\,\,\,\,\,\,\,\,\,\,u_{\alpha,r,p}:=\argmin\flp{\norm{u-u_\eta}_{L^2(Q)}^2+ \RVL^{r,\kappa}_{\alpha,p}(u)+\kappa H^{\lfloor r_1\rfloor+1}(u):\right.\\
&\,\,\,\,\,\,\,\,\,\,\,\,\,\,\,\,\,\,\,\,\,\,\,\,\,\,\,\,\,\,\,\,\,\,\,\,\,\,\,\,\,\,\,\,\,\,\,\,\,\,\,\,\,\,\,\,\,\,\,\,\,\,\,\,\,\,\,\,\,\,\,\,\,\,\,\,\,\,\,\,\,\,\,\,\,\,\,\,\,\,\,\,\,\,\,\,\,\,\,\,\,\,\,\,\,\,\,\,\,\,\,\,\,\,\,\,\,\,\,\,\,\,\,\left. u\in H_0^{\lfloor r_1\rfloor+1}(Q)}.\tag{$\mathcal T$-L2}\label{S_scheme_L2_D0_intro}&
\end{flalign}
with the expanded training ground defined in \eqref{big_training_set}, shall indeed provides an improved reconstruction image compared to scheme $\mathcal B$ and $\mathcal B'$.\\\\
We remark that the new regularizer $\RVL$ takes a \emph{modular design}. That is, each of the parameters $\alpha$, $p$, $r$ does not depend on another, 
and can be independently fine tuned to the current imaging processing task. For example, if the task is image denosing, then, based on our previous experience, 
the intensity parameter $\alpha$ and the derivation parameter $r$ play an essential role. 
Hence, we could opt to freeze the Euclidean parameter $p$, and optimize with respect to $\alpha$ and $r$, therefore reduce CPU consumption. 
On the other hand, if the task is inpainting, then the intensity parameter $\alpha$ is less important (and usually set to a small value), 
and we could freeze $\alpha$ and only optimize with respect to $r$ and $p$. Of course, to achieve optimal results, we could always train with respect to all the parameters, 
which, however, involves rather larger CPU cost.\\\\ 
The main result is:
\begin{theorem}[see Theorem \ref{thm_TS_S_solution}]
\label{main_thm_intro}
Let a Training Ground $\mathbb T$ be given. 
Then, the training scheme $\mathcal T$ (\eqref{S_scheme_L1_D0_intro}-\eqref{S_scheme_L2_D0_intro}) 
admits at least one solution $( \alpha_\T, r_\T,p_\T)\in \T$,
and provides a corresponding optimal reconstructed image $u_{\alpha_\T, r_\T,p_\T}$.
\end{theorem}
This paper is organized as follows. In Section \ref{tvs_setting_prop} we collect some preliminary notations and properties about the fractional order $TV^r$-seminorms,
mainly from the first work of our series (see \cite{liulud2019Image}). The construction and analysis of the new regularizer $\RVL$, 
as well as the new training scheme $\FT$, are the main topics of Section \ref{sec_bilevel}. Numerical implementations of some explicit examples are provided in Section \ref{num_sim_pd}.
\section{Notations and preliminary results}\label{tvs_setting_prop}
For future reference, we will use $r\in \R^+$ to denote a positive constant, and we write $r=\ir+s$, where $\ir$ denotes the integer part of $r$ and $s\in[0,1)$ denotes the fractional
part.
\subsection{Riemann-Liouville fractional order derivative}\label{sec_fractional_derivative}
Let $I:=(0,1)$ be the unit interval, and let $w\in C^\infty(I)$ be a given function. The \emph{left-sided R-L derivative} of order $r=\ir+s\in\R^+$ (see  \cite{MR1347689}) 
is defined pointwise by
\be\label{R_L_frac_1d_left}
d^{r}_{L}w(x) = \frac{1}{\Gamma(1-s)}\fsp{\frac d{dx}}^{\ir+1}\int_0^x \frac{w(t)}{(x-t)^{s}}dt, \qquad x\in I,
\ee 
where \[\Gamma(s):=\int_0^\infty e^{-t}t^{s-1}dt
             \]
denotes the Gamma function.             
 Similarly, the \emph{right-sided R-L derivative} and \emph{central-sided R-L derivative} of order  $r\in\R^+$ are defined 
 respectively by
\be
d^{r}_{R}w(x) = \frac{(-1)^{\ir+1}}{\Gamma(1-s)}\fsp{\frac d{dx}}^{\ir+1}\int_x^1 \frac{w(t)}{(t-x)^{s}}dt,
\ee
and
\be
d_C^{r}w(x) = \frac12 \fmp{d^{r}_{L}w(x)+(-1)^{\ir+1}d^{r}_{R}w(x)}.
\ee
The left and right \emph{Riemann-Liouville} fractional integrals of order $r\in\R^+$ are defined by
\[
(I_{L}^r w)(x):=\frac{1}{\Gamma(r)}\int_0^x \frac{w(t)}{(x-t)^{1-r}}dt, \]
and 
\[(I_{R}^r w)(x):=\frac{1}{\Gamma(r)}\int_x^1 \frac{w(t)}{(x-t)^{1-r}}dt
\] 
respectively. We also recall another type of fractional derivative, the left-sided \emph{Caputo} $r$-order derivative, given by 
\be
d^r_{L,c} w(x):=\frac{1}{\Gamma(1-s)}\int_0^x \frac{(d^{\ir+1})w(t)}{(x-t)^{s}}dt.
\ee
Similarly, the right-sided and central-sided \emph{Caputo} $r$-order derivative are given by
\[
d^{r}_{R,c}w(x):= \frac{(-1)^{\ir+1}}{\Gamma(1-s)}\int_x^1 \frac{(d^{\ir+1})w(t)}{(t-x)^{s}}dt,\]
 and \[d^{r}_{C,c}w(x):=\frac12 \fmp{d^{r}_{L,c}w(x)+(-1)^{\ir+1}d^{r}_{R,c}w(x)}\]
respectively. We also recall that the following integration by parts formula holds:
\be
\int_I w\,d_c^r \vp\,dx=(-1)^{\ir+1}\int_I d^rw\, \vp\,dx.
\ee
In view of \eqref{caputo_eq_RL}, we also have
\be
\int_I w\,d^r \vp\,dx=(-1)^{\ir+1}\int_I d^rw\, \vp\,dx.
\ee
Moreover, both types of fractional derivatives are linear: given $a$, $b\in\R$, we have
\be\label{linearity_frac}
d^r(a w_1(x)+b w_2(x))=a \,d^r w_1(x)+b\, d^r w_2(x)
\ee
for any functions $w_1$, $w_2$. We shall use this property repeatedly throughout this article. \\\\
We close Section \ref{sec_fractional_derivative} by citing the following one dimensional result. For convenience, we use a unified notation by writing $I^r = d^{-r}$, $r>0$.
\begin{remark}\label{constant_in_kernal}
We remark that (see \cite[Remark 2]{zhang2015total}), for $\vp\in C^\infty(I)$ with
zero boundary conditions on sufficiently high order derivatives, i.e. $v\lfloor_{\partial I}=dv\lfloor_{\partial I}=\cdots = d^\ir v\lfloor_{\partial I}=0$, we have
\be\label{caputo_eq_RL}
d^{r}_{C,c}\vp(x)= d^{r}_{C}\vp(x),\text{ for all }x\in I.
\ee
Also, we note that for any constant $M\in\R$, we have $d^{r}_{C,c} M=0$.
\end{remark}

\begin{theorem}[Semigroup properties in one dimension, {\cite[Theorem 2.6]{MR1347689}}]\label{semigroup_frac_int}
Let $r\in\R^+$ be given.
\begin{enumerate}[1.]
\item
The fractional order integration operators $I^r$ form a semigroup on $L^p(I)$, $p\geq 1$, which is strongly continuous for all $r\geq 0$,
and uniformly continuous for all $r> 0$. 
\item
More explicitly (see \cite[(2.72)]{MR1347689}) we have
\be
\norm{I^r w}_{L^1(0,1)}\leq \frac{1}{r\Gamma(r)}\norm{w}_{L^1(0,1)}.
\ee
\end{enumerate}
\end{theorem}

\subsection{Real order (an)-isotropic total variation}\label{sec_real_order_atv}
We first recall the definition of $\ell^p$-\emph{Euclidean} norm. Given a point $x=(x_1,\ldots,x_N)\in\rn$, and $p\in[1,+\infty)$, the $\ell^p$-Euclidean norm of $x$ is 
\be
\abs{x}_p = \fsp{\sum_{i=1}^N \abs{x_i}^p}^{1/p}\text{ and }\abs{x}_\infty=\max\flp{\abs{x_i}:\,\,i=1,\ldots,N}.
\ee 
Note that for $p=2$, it coincides with the standard Euclidean norm $\abs{x}=\abs{x}_2$.\\\\
We next introduce the definition of real order (an)-isotropic total variation.
\begin{define}\label{isotropic_frac_variation}
We define the \emph{$r$-order total variation} $TV_{\ell^p}^r(u)$ of a function
 $u\in L^1(Q)$ as follows.
\begin{enumerate}[1.]
\item
For $r=s\in(0,1)$ (i.e. $\ir=0$), we define
\be\label{RLFOD}
TV_{\ell^p}^s(u):=\sup\flp{\int_Qu\, \divg^s \vp \,dx:\,\,\vp\in C_c^\infty(Q;\rn)\text{ and }\abs{\vp}_{\ell^p}^\ast\leq 1},
\ee  
where 
\be\label{scaled_divg}
\divg^s u:=[(1-1/N)s+1/N]\sum_{i=1}^N\partial^s_i\vp_i. 
\ee
\item
For $r=\ir+s$ where $\ir\geq1$, we define 
\be\label{RLFODr}
TV_{\ell^p}^r(u):=\sup\flp{\int_Qu\, \divg^s[\divg^\ir \vp] \,dx:\,\,\vp\in C_c^\infty(Q;\M^{N\times (N^{\ir})})\text{ and }\abs{\vp}_{\ell^p}^\ast\leq 1}.
\ee
\end{enumerate}
\end{define}

\begin{remark}\label{an_iso_equ}
We note that the norms $\abs{\cdot}_p$, $1\le p<\infty$, are all equivalent on $\rn$. That is, for any $1\leq q<p\leq \infty$, we have
\be\label{equivalent_p_norm}
\abs{x}_{\ell^p}\leq\abs{x}_{\ell^q}\leq N^{1/q-1/p}\abs{x}_{\ell^p}.
\ee
\end{remark}
\begin{theorem}[Lower semi-continuity with respect to a fixed order $r\in\R^+$]\label{weak_star_comp_s}
Given $u\in L^1(Q)$, and a sequence $\seqn{u_n}\subset BV^r(Q)$, satisfying one of the following conditions:
\begin{enumerate}[1.]
\item
$\seqn{u_n}$ is locally uniformly integrable and $u_n\to u$ a.e.,
\item or
$u_n\wtos u$ in $\mb(Q)$.
\end{enumerate}
Then we have
\be\label{liminf_s_weak}
\liminf_{n\to\infty} TV^r(u_n)\geq TV^r(u).
\ee
\end{theorem}
\begin{theorem}[Approximation by smooth functions \cite{liulud2019Image}]\label{approx_smooth}
Given $r\in\R^+$ and $u\in BV^r(Q)\cap H_0^{\ir+1}(Q)$, then there exists a sequence $\seqn{u_n}\subset   C^\infty(\bar Q)$ such that 
\be
u_n\to u\text{ in }L^1(Q)\text{ and }\limn {TV^r}(u_n)={TV^r}(u).
\ee
\end{theorem}
We close this section with the following important compactness theorem.
\begin{theorem}[Compact embedding of real order bounded variation space, \cite{liulud2019Image}]\label{compact_lsc_r}
Given $p>1$, and sequences $\seqn{r_n}\subset \R^+$ and $\seqn{u_n}\subset L^1(Q)$, such that $r_n\to r\in\R^+$ and
\be
\sup\flp{\norm{u_n}_{L^p(Q)}+TV^{r_n}(u_n):\,\,n\in\N}<+\infty,
\ee
then the following two assertions hold:
\begin{enumerate}[1.]
\item
there exists $u\in BV^r(Q)$ and, up to a subsequence, $u_n\wto u$ in $L^p(Q)$ and
\be
\liminf_{n\to\infty} TV^{r_n}(u_n)\geq TV^r(u)
\ee
\item
If in addition we assume that $\norm{u_n}_{L^\infty(Q)}$ is uniformly bounded, then we have
\be
\text{ $u_n\to u$ strongly in $L^1(Q)$}.
\ee
\end{enumerate}
\end{theorem}
\subsection{Consideration for numerical implementation}\label{the_greatest_saver}
In this section we present several necessary restrictions on the boundary conditions,
to deal with singularities arising from the definition of fractional order derivative, as well as from the application of numerical implementations.
\subsubsection{Smoothness for regularizers}
For the numerical implementation of the
 underlying variational problem in \eqref{mumfordshahori233}, 
we follow the smoothness structure via \emph{Huber}-regularization,
proposed in \cite[Section 2.2]{reyes2015structure}.\\\\
The Huber-regularization is usually carried out by a $L^2$-penalty on the regularizer. That is, it induces a $H^1$ regularization on the 
$TV$ seminorm. In general, the regularizer is a convex, proper, and lower semicontinuous smoothing functional $H: L^1(Q)\to[0,+\infty]$, satisfying the following assumption.
\begin{assumption}[{see \cite[Assumptions A-H, Section 2.2]{reyes2015structure}}]
We assume that $0\in\operatorname{dom}H$, and for every $\delta>0$, there exists $u^\delta\in L^1(Q)$ such that
\be
H(u^\delta)<+\infty\text{ and } \mathcal I(u^\delta)\leq \mathcal I(u)+\delta.
\ee
\end{assumption}
\subsubsection{Boundary condition on the fractional order derivative}
We have seen from Section \ref{sec_fractional_derivative} that the RL fractional order derivatives require boundary conditions. 
In particular, if the function is not vanishing on the boundary, then singularities might arise in numerical simulation, since the computations at the inner nodes require such values. 
However, such conditions are often impractical in imaging applications, and inaccurate boundary conditions can easily lead to oscillations near the boundaries. 
Therefore, a proper treatment of the boundary conditions for problems involving fractional order derivatives is crucial.\\\\
In \cite[Section 4]{zhang2015total} the authors presented a method to reduce nonzero boundary conditions to zero boundary conditions, 
so that numerical algorithms become applicable. For reader's convenience, we report the one dimension construction. Given a function $u\in BV(I)$, with $I=(0,1)$ denoting the unit interval, such that 
\be
u(0)=a\text{ and }u(1)=b,
\ee
we introduce an auxiliary function 
\be
e(x):=a(1-x)+bx,
\ee
and set $\bar u(x)=u(x)-e(x)$. Thus, we have
\be
\bar u\lfloor_{\partial I}=0\text{ and }\bar u'\lfloor_{\partial I}=0,
\ee
where the zero Neumann boundary condition is imposed by artificially extending the boundary values. That is, $e'(0)=e'(1)=0$ on $\partial I$. 
The construction in two dimensions can be carried out in a similar spirit, after accurate estimates of $u(x_1,x_2)$ at the corners and edges are obtained. 
We refer to \cite[Section 4, Remark 4]{zhang2015total} for further details.\\\\
Later in this article, we shall also see that such boundary conditions are naturally compatible with the observations from Remark \ref{good_condition}.
\section{Generalization of $TV$-type regularization and $\Gamma$-convergence}\label{sec_bilevel}
Let $u_\eta$ and $u_c\in L^2(Q)$ be the corrupted and clean images respectively. For future reference, we will refer to such pairs $(u_c,u_\eta)$ as \emph{training pairs}. 
\begin{remark}
Based on the observations from Section \ref{the_greatest_saver}, we restrict our discussion to functions $u$ such that both $u$ and its derivatives, up to a maximal order based on the order of derivative used on $u$, are
vanishing on the boundary. 
For example, in \cite{zhang2015total}, the order $r\in\R^+$ of $TV^r(u)$ satisfied $r\in(1,2)$, and hence, only functions $u\in H_0^2(Q)$ are considered. In general, in this paper, we consider only function
\be\label{general_bas_image_assump}
\text{$u\in H_0^{\ir+1}(Q)$ for a given order of derivative $r\in\R^+$.}
\ee
We will see later that such zero boundary conditions are compatible with fractional order derivatives.
\end{remark}
\begin{remark}
The term that $\kappa {H^{\ir+1}(u)}$ is added as the Huber-regularization, where $\kappa>0$ is a (small) fixed parameter. 
The zero boundary condition is enforced by restricting $u$ to the space $H_0^{\ir+1}(Q)$. For example, when $1<r< 2$, we have $u\in H_0^2(Q)$, 
which gives the desired zero Dirichlet and Neumann boundary conditions, used in \cite{zhang2015total}. As $r$ increases, the zero boundary conditions given by the space $H_0^{\ir+1}(Q)$
naturally matches with the corresponding \emph{PDE} problem with order $2\ir$.
\end{remark}
We next recall the definition of representable functions.
\begin{define}[Representable functions]\label{frac_represent_I}
We denote by $\mathbb I^r(L^1(I))$, $r>0$, the space of functions $f$ represented by the $r$-order derivative of a summable function. That is,
\be
\mathbb I^r(L^1(I)):=\flp{f\in L^1(I):\,\, f=\mathbb I^rw,\,\,w\in L^1(I)}.
\ee
\end{define}
Next we recall several theorems on representable functions in one dimension, from \cite{MR1347689}.
\begin{theorem}\label{thm_MR1347689}
For reader's convenience, we use a unified notation, by writing $\mathbb I^r = d^{-r}$, $r>0$.
\begin{enumerate}[1.]
\item\label{cite_represent_frac}
{\cite[Theorem 2.3]{MR1347689}}
Condition $w(x)\in \mathbb I^r(L^1(I))$, $r>0$, is equivalent to
\be
\mathbb I^{\ir+1-r}[w](x)\in W^{\ir+1,1}(I),\qquad r=\ir+s 
\ee
and
\be
(d^l \mathbb I^{\ir+1-r}[w])(0)=0, \qquad l=0,1,\cdots,\ir.\label{bdy_frac_cond}
\ee
\item\label{MR1347689T2_5}
{\cite[Theorem 2.5]{MR1347689}}
Let $w\in L^1(I)$ be given. The relation
\be
\mathbb I^{r_1} \mathbb I^{r_2} w=\mathbb I^{{r_1}+r_2}w
\ee
is valid if one of the following conditions hold:
\begin{enumerate}[1.]
\item
$r_2>0$, ${r_1}+r_2>0$, provided that $w\in L^1(I)$,
\item
$r_2<0$, ${r_1}>0$, provided that $w\in \mathbb I^{-r_2}(L^1(I))$,
\item
${r_1}<0$, ${r_1}+r_2<0$, provided that $w\in \mathbb I^{-{r_1}-r_2}(L^1(I))$.
\end{enumerate}
\item 
{\cite[Theorem 2.6]{MR1347689}}
Let $r\in\R^+$ be given.
\begin{enumerate}[1.]
\item
The fractional order integral operators $\mathbb I^r$ form a semigroup on $L^p(I)$, $p\geq 1$, which is continuous in the uniform topology for all $r> 0$,
and strongly continuous for all $r\geq 0$. 
\item
It holds (see \cite[(2.72)]{MR1347689}) 
\be
\norm{\mathbb I^r w}_{L^1(a,b)}\leq (b-a)^r\frac{1}{r\Gamma(r)}\norm{w}_{L^1(a,b)}.
\ee
\end{enumerate}
\end{enumerate}
\end{theorem}

\begin{remark}\label{good_condition}
Most of the results from \cite{MR1347689} are about functions in $\mathbb I^r(L^1(I))$. In view of 
Theorem \ref{thm_MR1347689}, Assertion 1,
this is equivalent to assuming \eqref{bdy_frac_cond}.
 A possible sufficient condition is, for instance, \eqref{general_bas_image_assump}.
\end{remark}

\subsection{Extension with primal extension}\label{sec_extension_primal}
In the following, by ``primal'' we will mean ``directly on $u$''.

\subsubsection{Extension of $TV$ with real order derivative}\label{sec_bilevel_TVr}
Let $r\in\R^+$ be given, and write $r=\ir+s$. We recall the definition of $r$-order total variation:
\be
TV^r(u):=\sup\flp{\int_Qu\, [\divg^s\divg^\ir \vp] \,dx:\text{ for }\vp\in C_c^\infty(Q;\M^{N\times (N^{\ir})})\text{ and }\abs{\vp}\leq 1}.
\ee
Next, we introduce the following functional $I_{\alpha,r}^\kappa(u)$.

\begin{define}
\label{def:fractional TVr functional}
Let $r\in[0,+\infty)$ and $\alpha\in \R^+$ be given. We define the functional $\I_{\alpha,r}^\kappa:L^1\to [0,+\infty]$ by
\be
\mathcal I_{\alpha,r}^\kappa(u):=
\begin{cases}
\norm{u-u_\eta}_{L^2(Q)}^2+\alpha TV^{r}(u)+\kappa H^{\ir+1}(u)&\text{ for }u\in H_0^{\ir+1}(Q),\\
+\infty&\text{ otherwise }.
\end{cases}
\ee
\end{define}

We first show that for every given $u_\eta\in L^2(Q)$, the minimizing problem associated with $\mathcal I^\kappa_{\alpha,r}$ admits a unique solution.

\begin{proposition}\label{unique_exist_lower}
Let $u_\eta\in L^2(Q)$ and $(\alpha,r)\in\T$ be given. Then, there exists a unique $u_{\alpha,\B}\in H_0^{\ir+1}(Q)$ such that
\be
u_{\alpha,r}\in \argmin\flp{\mathcal I_{\alpha,r}^\kappa(u):\, u\in H^{\ir+1}_0(Q)}.
\ee
\end{proposition}

\begin{proof}
Without loss of generality, we can assume $\alpha=1$. Let $\seqn{u_n}\subset H_0^{\ir+1}(Q)$ be a sequence satisfying
\be\label{goes_to_hell_tvr}
\mathcal I_{\alpha,r}^\kappa(u_n)\leq \inf\flp{\mathcal I_{\alpha,r}^\kappa(u):\, u\in H^{\ir+1}_0(Q)}+1/n.
\ee
Thus, there exists a constant $C$ such that 
\begin{align*}
\norm{u_n}^2_{L^2(Q)}+H^{\ir+1}(u_n)&\leq \norm{u_n-u_\eta}_{L^2(Q)}^2+\norm{u_\eta}_{L^2(Q)}+H^{\ir+1}(u_n)\\
&\leq  \mathcal I_{\alpha,r}^\kappa(u_n)+\norm{u_\eta}_{L^2(Q)}^2\leq C<+\infty.
\end{align*}
By Sobolev inequality, we deduce that 
\be
\sup\flp{\norm{u_n}_{W^{\ir+1,2}}:\,\,n\in\N}<+\infty.
\ee
Thus, up to a subsequence, there exists $u\in W^{\ir+1,2}(Q)$ such that 
\be\label{hir12_weak_zero}
u_n\to u\text{ strongly in }L^1(Q)\text{ and }u_n\wto u\text{ weakly in }W^{\ir+1,2}(Q).
\ee
Moreover, by the lower semicontinuity of the norm $\|\cdot\|_{H^{\ir+1}(Q)}$, we have $u\in H_0^{\ir+1}(Q)$. Next, by Theorem \ref{approx_smooth}, we have
\be
\liminfn\,TV^r(u_n)\geq TV^r(u).
\ee
Combined with \eqref{hir12_weak_zero} gives
\be
\liminfn\,\mathcal I_{\alpha,r}^\kappa(u_n)\geq \mathcal I_{\alpha,r}^\kappa(u),
\ee
which, in view of \eqref{goes_to_hell_tvr}, concludes the proof.
\end{proof}

\begin{proposition}[$\Gamma$-convergence of $\mathcal I_{\alpha,r}^\kappa$]
\label{thm_gamma_tvr}
Given a sequence $\seqn{(\alpha_n,r_n)}\subset \T$ such that $(\alpha_n,r_n)\to(\alpha_0,r_0)\in\T$, 
then the functional $\mathcal I_{\alpha_n,r_n}^\kappa$ $\Gamma$-converges to $\mathcal I_{\alpha_0,r_0}^\kappa$ in the $L^1(Q)$ topology. That is, for any $u\in L^1(Q)$ the following two conditions hold:
\begin{enumerate}
\item[{\rm (LI)}] For all sequences $\{u_n\}$ such that 
\be u_n\to u\text{ in }L^1(Q),\ee
we have
\be
\I_{\alpha_0,r_0}^\kappa(u)\leq \liminf_{n\to +\infty}\I_{\alpha_n,r_n}^\kappa(u_n).\ee
\item[{\rm (RS)}]
For each $u\in BV(Q)$, there exists a sequence $\seqn{u_n}\subset BV^{r_n}(Q)\cap H_0^{\lfloor r_n\rfloor+1}$ such that 
\be u_n\to u\text{ in }L^1(Q),\ee
and 
\be
\limsup_{n\to +\infty}\,\I_{\alpha_n,r_n}^\kappa(u_n)\leq \I_{\alpha_0,r_0}^\kappa(u).\ee
\end{enumerate}
\end{proposition}


The proof of Proposition \ref{thm_gamma_tvr} is split into several steps.
\begin{proposition}\label{asymptotic_tvs}
Given $r\in \R^+$, $r_n\to r$, and $u\in BV^r(Q)\cap H_0^{\ir+1}(Q)$, the following two assertions hold.
\begin{enumerate}[1.]
\item
There exists $\seqk{u_k}\subset BV^{r_k}(Q)\cap H_0^{\ir+1}(Q)$ such that $u_k\to u$ strongly in $L^1(Q)$ and
\be
\lim_{r_k\to r} TV^{r_k}(u_k)= TV^{r}(u) 
\ee
\item
If in addition we assume $u\in C^\infty(\bar Q)$, we have also
\be
\lim_{r_k\to r} TV^{r_k}(u)=TV^{r}(u).
\ee
\end{enumerate}
\end{proposition}
\begin{proof}
We claim separately that 
\begin{enumerate}[1.]
\item
for any given sequence $\seqk{r_k}\subset \R^+$ such that $r_k\to r_0$, it holds
\be
\liminf_{r_k\to r_0} TV^{r_k}(u)\geq TV^{r_0}(u).
\ee
\item
For any given sequence $\seqk{r_k}\subset \R^+$ such that $r_k\to r_0$, then, up to a subsequence,
 $u_k\to u$ strongly in $L^1(Q)$, and
\be
\limsup_{r_k\to r_0} TV^{r_k}(u_k)\leq TV^{r_0}(u).
\ee
\end{enumerate}
Statement 1 can be deduced directly from the definition, and basic properties, of $TV^r$. 
We focus on showing Statement 2. We split our argument into two cases.\\\\ 
\underline{Case 1.} Assume $r_0\in\R^+\setminus\N$. In this case, the proof of Assertion 2 does not rely on the boundary conditions on $u$. 
Assume first $u\in C^\infty(\bar Q)$, and let $\delta>0$ be given. Then for each $k\in\N$, we could find $\vp_k\in C_c^\infty(Q)$ such that
\be\label{moving_upper_test}
TV^{r_k}(u)\leq \int_Q u\,\divg^{r_k}\vp_k\,dx+\delta.
\ee
On the other hand, since $u\in C^\infty(\bar Q)$, we have 
\be
\abs{D^r u(x)}\leq \norm{\nabla^{\ir+1}u}_{L^\infty(Q)}\frac 1{(1-s)\Gamma(1-s)},
\ee
where $r=\ir+s$. Hence, since $r_k\to r_0\in\R^+\setminus \N$, we have 
\be
\sup_{k}\abs{D^{r_k} u(x)}\leq 2\norm{\nabla^{\lfloor r_0\rfloor+1}u}_{L^\infty(Q)}\frac 1{(1-s_0)\Gamma(1-s_0)}<+\infty,
\ee
which then gives
\be\label{L1upperbound}
 \abs{\int_Q \nabla^{r_k}u\,\vp_k\,dx}
\leq \int_Q\abs{\nabla^{r_k} u(x)}dx \leq \int_Q 2\norm{\nabla^{\lfloor r_0\rfloor+1}u}_{L^\infty(Q)}\frac 1{(1-s_0)\Gamma(1-s_0)}dx+\delta<+\infty.
\ee
Also, since $u\in C^\infty(\bar Q)$ and $r_0\in\R^+\setminus \N$, we have
\be
\nabla^{r_k} u(x)\to \nabla^{r_0}u(x),\text{ for a.e. }x\in Q.
\ee
This, combined with \eqref{L1upperbound}, allows us to apply the dominated convergence theorem, to conclude that 
\be
\lim_{k\to\infty}\int_Q\abs{\nabla^{r_k} u(x)}dx =\int_Q\abs{\nabla^{r_0}u}dx.
\ee
Together with \eqref{moving_upper_test}, we have
\begin{align*}
\limsup_{k\to\infty}TV^{r_k}(u) &\leq \limsup_{k\to\infty}\int_Q u\,\divg^{r_k}\vp_k\,dx+\delta\leq \limsup_{k\to\infty}\abs{\int_Q \nabla^{r_k}u\,\vp_k\,dx}+\delta\\
&\leq\limsup_{k\to\infty} \int_Q\abs{\nabla^{r_k} u}dx +\delta = \int_Q\abs{\nabla^{r_0}u}dx+\delta.
\end{align*}
By the arbitrariness of $\delta>0$, we conclude that 
\be\label{smooth_result}
\limsup_{k\to\infty}TV^{r_k}(u)\leq TV^{r_0}(u), \text{ for each }u\in C^\infty(\bar Q),
\ee 
as desired.\\\\
Now, assume $u\in BV^{r_0}(Q)$ only. By Theorem \ref{approx_smooth}, there exists a sequence $\seqe{u_\e}\subset BV^{r_0}(Q)\cap C^\infty(\bar Q)$ such that $u_\e\to u$ strongly in $L^1(Q)$ and
\be
TV^{r_0}(u_\e)\to TV^{r_0}(u),\text{ or }TV^{r_0}(u_\e)\leq TV^{r_0}(u)+O(\e).
\ee
This, combined with \eqref{smooth_result}, gives that, for each fixed $\e>0$, 
\be
\limsup_{k\to\infty}TV^{r_k}(u_\e)\leq TV^{r_0}(u_\e)\leq TV^{r_0}(u)+O(\e).
\ee
Thus, by a diagonal argument, there exists a (not relabeled) subsequence $u_{k_\e}$ such that 
\be
\limsup_{\e\to 0}TV^{r_{k_\e}}(u_\e)\leq TV^{r_0}(u),
\ee 
concluding the proof for this case.\\\\
\underline{Case 2:} Assume $r_0\in\N$. Since $u\in H^{\ir+1}_0(Q)$, we extend $u$ to all of $\rn$, by setting $u=0$ outside of $Q$, and let 
\be
u_\e(x):=u(x_\e)\text{ for }x_\e:=\frac1{1+\e}[x-(1/2)^N] + (1/2)^N,\,\,x\in\flp{x\in\rn:\,\,\operatorname{dist}(x,Q)\leq \e}.
\ee
Then, we have $u_\e(x)$ is compactly supported in $Q$, and in view of zero boundary condition on $u$, we have that 
\be\label{compactlized_eq}
TV^{r}(u_\e)\to TV^{r}(u).
\ee 
In view of Theorem \ref{approx_smooth}, it is not restrictive to impose $u_\e\in C_c^\infty(Q)$. Hence, 
we have $\nabla^{r_n}u_\e\to\nabla^r u_\e$ for a.e. $x\in Q$. Then, by the dominated convergence theorem, we conclude that, for each $\e>0$,
\be
TV^{r_n}(u_\e)\to TV^r(u_\e).
\ee
This, combined with \eqref{compactlized_eq}, gives a sequence $u_{\e_n}$ such that $u_{\e_n}\to u$ in $L^2(Q)$ and 
\be
TV^{r_n}(u_{\e_n})\to TV^r(u),
\ee
as desired.
\end{proof}


\begin{proposition}[$\Gamma$-$\liminf$ inequalities]\label{compact_semi_para}
Let $\seqn{(\alpha_n,r_n)}\subset \mathbb \R^2$ be a sequence satisfying $r_n\to r_0\in\R^+$, and $\alpha_n\to \alpha_0\in\R^+$. For every $n\in \N$, let $u_n\in BV^{r_n}(Q)$ be such that 
\be
\sup\flp{\mathcal I_{\alpha_n,r_n}^\kappa(u_n):\,\,n\in\N}<+\infty.
\ee
Then, there exists $u\in BV^{r_0}(Q)$ such that, up to a (non-relabeled) subsequence,
\be
u_n\to u\text{ strongly in }L^1(Q),
\ee
and 
\be
\liminfn \,\mathcal I_{\alpha_n,r_n}^\kappa(u_n)\geq \mathcal I_{\alpha_0,r_0}^\kappa(u).
\ee
\end{proposition}
\begin{proof}
The prof can be directly inferred from Proposition \ref{compact_lsc_r}, since we have $\alpha\in\R^+$.
\end{proof}

%

%
%
%

%

\subsubsection{Extending with underlying Euclidean norm}\label{sec_bilevel_TVp}
Let $p\in[1,+\infty]$ and $r\in[0,+\infty)$ be given. We define the $\ell^p$-(an)-isotropic real $r$-order total variation $TV^r_{\ell^p}$ by
\be\label{def_tv_pr}
TV_{\ell^p}^r(u):=\sup\flp{\int_Qu\, \divg^s[\divg^\ir \vp] \,dx:\,\,\vp\in C_c^\infty(Q;\M^{N\times (N^{\ir})})\text{ and }\abs{\vp}_{\ell^p}^\ast\leq 1}.
\ee
\begin{lemma}\label{TVlpq_lemma}
Given $1\leq q<p\leq \infty$, we have
\be
N^{1/q-1/p}TV_{\ell^p}^r(u)\leq TV_{\ell^q}^r(u)\leq TV_{\ell^p}^r(u),
\ee
for all $r\in\R^+$ and $u\in BV^r(Q)$.
\end{lemma}
\begin{proof}
Let $1\leq q<p\leq+\infty$ be given. From Remark \ref{an_iso_equ} we have 
\be\label{lp_euclidean_use}
N^{1/p-1/q}\abs{\vp(x)}_{\ell^{p^\ast}}\leq \abs{\vp(x)}_{\ell^{q^\ast}}\leq \abs{\vp(x)}_{\ell^{p^\ast}},
\ee
for all $\vp\in C_c^\infty(Q;\rn)$, and $x\in Q$. That is, for any $\vp\in C_c^\infty(Q;\rn)$ such that $\abs{\vp(x)}_{\ell^{p^\ast}}\leq 1$, 
we have $\abs{\vp(x)}_{\ell^{q^\ast}}\leq 1$, and combined with \eqref{def_tv_pr}, gives
\be
TV^r_{\ell^q}(u)\leq TV^r_{\ell^p}(u),
\ee
for any $u\in BV^r(Q)$. Similarly, we use the left hand side of \eqref{lp_euclidean_use} to conclude that 
\be
N^{1/q-1/p}TV_{\ell^p}^r(u)\leq TV_{\ell^q}^r(u),
\ee
as desired.
\end{proof}
\begin{define}
Let $r\in[0,+\infty)$ and $\alpha\in \R^+$ be given. Let
\be
\mathcal I^\kappa_{\alpha,r,p}(u):=
\begin{cases}
\norm{u-u_\eta}_{L^2(Q)}^2+\alpha TV_{\ell^p}^{r}(u)+\kappa H^{\ir+1}(u)&\text{ for }u\in H_0^{\ir+1}(Q),\\
+\infty&\text{ otherwise }.
\end{cases}
\ee
\end{define}
\begin{proposition}[$\Gamma$-convergence of $\mathcal I_{\alpha,r,p}^\kappa$ functional]
\label{thm:new-Gamma}
Let $\seqn{(\alpha_n,r_n,p_n)}\subset \T$ be a sequence such that $(\alpha_n,r_n,p_n)\to(\alpha_0,r_0,p_0)\in\T$. Then, 
the functional $\mathcal I_{\alpha_n,r_n,p_n}^\kappa$ $\Gamma$-converges to $\mathcal I_{\alpha_0,r_0,p_0}^\kappa$ in the $L^1(Q)$ topology. That is, 
for every $u\in L^1(Q)$ the following two conditions hold:
\begin{enumerate}
\item[{\rm (LI)}] If 
\be u_n\to u\text{ in }L^1(Q),\ee
then 
\be
\I_{\alpha_0,r_0,p_0}^\kappa(u)\leq \liminf_{n\to +\infty}\I_{\alpha_n,r_n,p_n}^\kappa(u_n).\ee
\item[{\rm (RS)}]
For each $u\in BV^r(Q)\cap H_0^{\ir+1}(Q)$, there exists $\seqn{u_n}\subset BV^{r_n}(Q)\cap H_0^{\lfloor r_n\rfloor+1}$ such that 
\be u_n\to u\text{ in }L^1(Q),\ee
and 
\be
\limsup_{n\to +\infty}\,\I_{\alpha_n,r_n,p_n}^\kappa(u_n)\leq \I_{\alpha_0,r_0,p_0}^\kappa(u).\ee
\end{enumerate}
\end{proposition}
\begin{proof}
We prove the $\liminf$ inequality first. Consider a sequence $p_n\to p$. From Lemma \ref{TVlpq_lemma}, we have, for each $n\in\N$, 
\be
TV^{r_n}_{\ell^{p_n}}(u_n)\geq N^{-\abs{1/p_n-1/p}}TV^{r_n}_{\ell^p}(u_n),
\ee
and in view of (LI) from Proposition \ref{thm_gamma_tvr}, it gives
\be
\liminfn\,TV^{r_n}_{\ell^{p_n}}(u_n)\geq \fsp{\liminfn\,N^{-\abs{1/p_n-1/p}}}\fsp{\liminfn\,TV^{r_n}_{\ell^p}(u_n)}\geq TV_{\ell^p}^r(u).
\ee

We analyze this proposition under the assumption $\alpha_n=1$ for all $n\in\N$. The thesis for a general sequence $\seqn{\alpha_n}$ follows by straightforward adaptations.\\\\
Fix $\e>0$. We first assume $p_n\searrow p$. That is, $p_n^\ast\nearrow p^\ast$. In view of \eqref{def_tv_pr},
we may choose $\vp_n\in C_c^\infty(Q;\R^2)$ such that $\abs{\vp_n}_{p_n^\ast}\leq 1$ and
\be\label{xiaoyidian_yidian}
TV_{p_n}(u)\leq \int_Q u\,\divg\vp_n\,dx+\e.
\ee 
Since $p_n^\ast\nearrow p^\ast$, we have $\abs{\vp_n}_{p^\ast}\leq \abs{\vp_n}_{p_n^\ast}\leq1$, and
\be
\int_Q u\,\divg\vp_n\,dx\leq TV_p(u)
\ee
for each $n\in\N$. This, together with \eqref{xiaoyidian_yidian}, gives
\be\label{yidiandianzengjia}
\limsup_{n\to\infty}TV_{p_n}(u)\leq \limsup_{n\to\infty}\int_Q u\,\divg\vp_n\,dx+\limsup_{n\to\infty}\e\leq  \limsup_{n\to\infty}TV_p(u)+\e = TV_p(u)+\e,
\ee
and we conclude by taking the limit $\e\searrow 0$.\\\\
Now we consider the case $p_n\nearrow p$, i.e., $p_n^\ast\searrow p^\ast$. 
We take again a sequence $\seqn{\vp_n}\subset C_c^\infty(Q;\R^2)$ such that $\abs{\vp_n}_{p_n^\ast}\leq 1$ for each $n\in\N$, and \eqref{xiaoyidian_yidian} holds.
In view of \eqref{equivalent_p_norm}, for each $n$ we have
\be
N^{1/p_n^\ast-1/p^\ast}\abs{\vp_n}_{p^\ast}\leq \abs{\vp_n}_{p_n^\ast}\leq1.
\ee
That is, 
\be
N^{1/p_n^\ast-1/p^\ast}\int_Q u\,\divg\vp_n\,dx\leq TV_p(u).
\ee
This, combined with \eqref{yidiandianzengjia}, gives
\be
\limsup_{n\to\infty}TV_{p_n}(u)\leq \limsup_{n\to\infty}N^{1/p_n^\ast-1/p^\ast} TV_p(u)+\e = TV_p(u)+\e,
\ee
and we conclude this proposition by letting $\e\searrow 0$.
\end{proof}

\subsection{Extension with infimal convolution}\label{sec_extension_infimal}
This extension is done by adding some auxiliary functions. We start by reviewing the definition of \emph{total generalized variation}, namely the $TGV$ seminorm.\\\\
For a given function $u\in L^1(Q)$, we define the \emph{second order total generalized variation} $TGV_\alpha^2$ (where $\alpha=(\alpha_0,\alpha_1)\in\R_+^2$) by
\be
TGV_\alpha^2(u):=\inf\flp{\alpha_0\abs{\nabla u-v}_\mb+\alpha_1\abs{\mathcal Ev}_\mb:\,\, v\in BD(Q;\rn)},
\ee
where $\mathcal Ev:=(\nabla v+(\nabla v)^T)/2$ denotes the symmetric derivative of $v\in L^1(Q;\rn)$. Incorporating with the Huber-regularization introduced in Section \ref{the_greatest_saver}, we define the $TGV^{2,\kappa}_\alpha$ seminorm by  
\be
TGV_\alpha^{2,\kappa}(u):=\min\flp{\alpha_0\abs{\nabla u-v}_\mb+\alpha_1[\abs{\mathcal Ev}_\mb+\kappa H^1(v)]:\,\, v\in H_0^1(Q;\rn)},
\ee
where the zero Dirichlet boundary condition on $v$ is imposed to enforce the zero Neumann boundary condition on $u$.\\\\
Similarly, we could define the \emph{non-symmetric second order total generalized variation} $NsTGV_\alpha^{2,\kappa}$ with Huber-regularization by
\be
NsTGV_\alpha^{2,\kappa}(u):=\min\flp{\alpha_0\abs{\nabla u-v}_\mb+\alpha_1[\abs{\nabla v}_\mb+\kappa H^1(v)]:\,\, v\in H_0^1(Q;\rn)}.
\ee
We remark that $NsTGV$ is known to provide, in general, more accurate results compared to $TGV^2_{\alpha}$, but with a higher computational cost.
For more properties of $TGV$ and $NsTGV$, we refer to \cite{valkonen2013total}. \\\\
Also, we could further extend $NsTGV^{2,\kappa}$ and $TGV^{2,\kappa}$ to higher order  $NsTGV^{k,\kappa}$ and $TGV^{k,\kappa}$, $k\in\N$, via 
\begin{multline}
NsTGV_{\alpha}^{k,\kappa}(u):=\inf\flp{\alpha_0\abs{\nabla  u- v_0}_{\mathcal{M}_b(Q)}+\alpha_1\abs{\nabla  v_0-v_1}_{{\mathcal{M}_b(Q)}}+\right.\\
\left. \cdots+\alpha_{k-1}\abs{\nabla v_{k-2}-v_{k-1}}_{{\mathcal{M}_b(Q)}}+\alpha_k[TV(v_{k-1})+\kappa H^1(v_{k-1})],\right.\\
\left. v_i\in BV(Q;\mathbb M^{N^{l}})\text{ for }0\leq i\leq k-2,\,\,v_{k-1}\in H_0^1(Q;\mathbb M^{N^\ir})},
\end{multline}
and similarly in $TGV^{k,\kappa}_\alpha$ by replacing $\nabla v_l$ with $\mathcal Ev_l$, for $l=0,\ldots, k-1$. 
\subsubsection{The real $r$-order $TGV^r$ seminorm} \label{sec_tgv_r}
Let the Huber-regularization parameter $0<\kappa\ll1$ be given, we denote by $\Pi$ the collection of seminorms
\be\label{use_liu2016Image}
\Pi:=\flp{TV^s(v)+\kappa H^1(v):\,\, 0\leq s\leq 1}.
\ee
In next proposition we claim that the collection $\Pi$ satisfies the properties defined in \cite[Definition 3.4]{liu2016Image}, 
provided we are under \cite[Assumptions 3.2 \& 3.3]{liu2016Image}.

\begin{proposition}\label{assum_basic_B}
We collect several properties regarding to seminorm $TV^s(u)$ and function $v\in BV^s(Q;\rn)\cap H_0^1(Q;\rn)$ with zero boundary conditions.
\begin{enumerate}[1.]
\item
The null space of the seminorm $TV^s(\cdot)+\kappa H^1(\cdot)$ has finite dimension.
\item
For every $v\in BV^s(Q;\rn)\cap H^1_0(Q;\rn)$ there exists $\seqn{u_n}\subset C^\infty(\bar Q,\rn)$ such that 
\be
u_n\to u\text{ strongly in }L^1(Q;\rn)\text{ and }TV^s(v_n)+H^1(v_n)\to TV^s(v)+H^1(v);
\ee
\item
Let $\seqn{v_n}\subset H_0^1(Q;\rn)$ and $\seqn{s_n}\subset [\sigma,1-\sigma]$ be such that 
\be\label{eq_assum_basic_BPi}
\sup\flp{\norm{v_n}_{L^1(Q;\rn)}+TV^{s_n}(v_n)+\kappa H^1(v_n):\,\, n\in\N}<+\infty.
\ee
Then there exist $v\in BV^s(Q;\R^N)\cap H_0^1(Q;\rn)$ such that, up to a subsequence (not relabeled), 
\be\label{eq_v0_unif}
v_n\to v\text{ strongly in }L^1(Q;\R^N),
\ee
and 
\be\label{eq_B_n_uniform_bdd2}
\liminfn \,TV^{s_n}(v_n)+\kappa H^1(v_n)\geq TV^s(v)+ \kappa H^1(v).
\ee
\end{enumerate}
\end{proposition}
\begin{proof}
Assertion 1 follows directly from the definition of $H^1$, and Assertion 2 can be deduced from Theorem \ref{approx_smooth}, since $v\in H^1_0(Q;\rn)$ has zero boundary conditions. \\\\
For Assertion 3, let $\seqn{v_n}\subset H_0^1(Q;\rn)$ be a sequence satisfying \eqref{eq_assum_basic_BPi}. 
Then, by the compact embedding properties of $H_0^1$, we have the existence of $v\in H_0^1(Q)$ such that \eqref{eq_v0_unif} holds. 
Finally, by Theorem \ref{weak_star_comp_s}, we conclude \eqref{eq_B_n_uniform_bdd2}, completing the proof.
\end{proof}

We next introduce the real-$r$-order total generalized variation $TGV^{r,\kappa}$, with the embedded Huber-regularization. Recall for any $r\in\R^+$, we write $r=k+s$, with $k=\ir$ and $s\in[0,1)$.

\begin{define}[The $TGV^{r,\kappa}$ seminorms]\label{TGV_fractional}
Let $r=k+s\in\R^+$ be given, and let $\alpha=(\alpha_0,\alpha_1,\alpha_2,\ldots, \alpha_{k})\in\R_+^{k+1}$. 
For every $u\in L^1(Q)$, we define its real order $TGV^{r,\kappa}$ seminorm as follows.

\begin{enumerate}[C{a}se 1.] 
\item
for $k=1$, i.e. $\alpha=(\alpha_0,\alpha_1)\in\R_+^2$, set
\begin{multline}
NsTGV_{\alpha}^{1+s,\kappa}(u)\\
:=\inf\flp{\alpha_0\abs{\nabla u- sv_0}_{\mathcal{M}_b(Q)}+\alpha_1s [TV^s(v_0)+\kappa H^1(v_0)]:\,v_0\in  H_0^1(Q;\rn)}.
\end{multline}
\item
for $k>1$, set
\begin{multline}\label{eq:def-tgv-B_k}
NsTGV_{\alpha}^{k+s,\kappa}(u):=\inf\flp{\alpha_0\abs{\nabla  u- v_0}_{\mathcal{M}_b(Q)}+\alpha_1\abs{\nabla  v_0-v_1}_{{\mathcal{M}_b(Q)}}+\right.\\
\left. \cdots+\alpha_{k-1}\abs{\nabla v_{k-2}-sv_{k-1}}_{{\mathcal{M}_b(Q)}}+\alpha_{k}s[TV^s(v_{k-1})+\kappa H^1(v_{k-1})],\right.\\
\left. v_i\in BV(Q;\mathbb M^{N^{l}})\text{ for }0\leq i\leq k-2,\,\,v_{k-1}\in H_0^1(Q;\mathbb M^{N^\ir})}.
\end{multline}
\end{enumerate}
Moreover, we say that $u$ belongs to the space of \emph{functions with $r$-order bounded total generalized variation}, and we write $u\in BGV_{\alpha}^{r}(Q)$, if
\be
\norm{u}_{BGV_{\alpha}^{r}(Q)}:=\norm{u}_{L^1(Q)}+{TGV_{\alpha}^{r}(u)}<+\infty,
\ee
where we again $r=k+s$, with $0\leq s<1$, $k=\ir\in\mathbb N$, $\alpha=(\alpha_0,\alpha_1,\alpha_2,\ldots, \alpha_{k})\in\R_+^{k+1}$. 
Additionally, we write $u\in BGV^{r}(Q)$ if there exists $\alpha\in \R_+^{k+1}$ such that $u\in BGV_{\alpha}^{r}(Q)$. Note that if $u\in BGV_{\alpha}^{k+s}(Q)$ for some $\alpha\in \R_+^{k+1}$, 
then $u\in BGV_{\beta}^{k+s}(Q)$ for every $\beta \in \R_+^{k+1}$.
\end{define}
\begin{remark}
Definition \ref{TGV_fractional} matches with the general framework used in \cite[Definition 3.6]{liu2016Image}. 
Moreover, since the collection $\Pi$, defined in \eqref{use_liu2016Image}, satisfies \cite[Assumptions 3.2 \& 3.3]{liu2016Image}, 
most of the results presented in \cite[Section 3, 4, and 5]{liu2016Image} are still valid. We collect the relevant results in the next proposition. 
\end{remark}
\begin{proposition}\label{weak_equi_norm}
Recall $TGV^{r,\kappa}$ defined in Definition \ref{TGV_fractional}, with collection $\Pi$ from \eqref{use_liu2016Image}.
\begin{enumerate}[1.]
\item
\cite[Proposition 3.8]{liu2016Image} Let $u\in L^1(Q)$ be given, and recall the definition of $TGV^{r,\kappa}(u)$ from Definition \ref{TGV_fractional}. 
Then, for every $r\geq 1$, $TGV^{r,\kappa}(u)<+\infty$ if and only if $u\in BV(Q)$. In particular, 
\be\label{eq_tgv_les_tv}
TGV^{r,\kappa}(u)\leq TV(u).
\ee
\item
\cite[Proposition 3.9]{liu2016Image} Let $u\in BV(Q)$ be given. Then the infimum in \eqref{eq:def-tgv-B_k} is attained by a unique 
function $ v=(v_0,\ldots,v_{k-1})$, with $v_l\in BV(Q;\mathbb M^{N^{l+1}})$ for $l=0,\ldots, k-2$ and $v_{k-1}\in BV^s(Q;\mathbb M^{N_k})$.
\item (Asymptotic behavior \cite[Proposition 3.10]{liu2016Image})\label{intermediate}
For all $u\in BV(Q)$ and $s_n\to s\in(0,1)$, up to a (non-relabeled) subsequence, it holds 
\be
\lim_{s_n\to s}{TGV^{k+s_n,\kappa}_{\alpha}(u)}= {TGV^{k+s,\kappa}_{\alpha}(u)}.
\ee
\end{enumerate}
\end{proposition}

Note that the asymptotic behavior provided in Statement 3, Proposition \ref{weak_equi_norm} only allows sequences $s_n\to s\in(0,1)$, i.e., $s\neq 0,1$. We study those two boundary cases in the following proposition.

\begin{proposition}\label{thm_asymptitic_PGV}
For every $u\in BV(Q)$ and $s_n\to s\in\flp{0,1}$, up to a (non-relabeled) subsequence, it holds 
\be
\lim_{s_n\to s}{TGV^{k+s_n,\kappa}_{\alpha}(u)}= {TGV^{k+s,\kappa}_{\alpha}(u)}.
\ee
\end{proposition}
\begin{proof}
We assume that $r\in[1,2]$. Then case $r\geq 2$ can be dealt analogously. 
Also, since $\alpha\in\R^2$ is fixed in this argument, for brevity we write $TGV$ instead of $TGV_\alpha$. \\\\
We write $r_n=1+s_n$, $s_n\in[0,1]$. Consider the case $s_n\searrow 0$ first, and by \eqref{eq_tgv_les_tv} we get
\be
\limsup_{s_n\searrow 0}\,TGV^{1+s_n,\kappa}(u)\leq  TV(u).
\ee
We next show the $\liminf$ inequality. From Proposition \ref{weak_equi_norm}, Assertion 2, we have a sequence $\seqn{v_0^n}\subset C^\infty(Q)$, such that, for each $n\in\N$, $T[v_0^n]=0$, and
\begin{align*}
\abs{s_nv_n}_{\mb}&+\abs{\nabla^{s_n}(s_nv_0^{s_n})}_{\mb(Q)} +\kappa \norm{\nabla (s_n v_0^{s_n})}_{L^2(Q)} \\
&\leq \abs{\nabla u-s_nv_n}_{\mb(Q)}+\abs{\nabla u}_{\mb(Q)}+\abs{\nabla^{s_n}(s_nv_0^{s_n})}_\mb+\kappa \norm{\nabla (s_n v_0^{s_n})}_{L^2(Q)} \\
&\leq 2TV(u)<+\infty.
\end{align*}
Thus, there exists $\bar v\in L^1(Q)$ such that $s_nv_0^{s_n}\to \bar v_0$, and
\be
\liminfn\,\abs{\nabla^{s_n}(s_nv_0^{s_n})}_{\mb(Q)}+\kappa \norm{\nabla (s_n v_0^{s_n})}_{L^2(Q)} \geq \norm{\bar v_0}_{L^1(Q)}.
\ee
Thus, 
\begin{align}
\abs{\nabla u}_{\mb(Q)}&\leq \abs{\nabla u-\bar v_0}_{\mb(Q)}+\norm{\bar v_0}_{L^1(Q)}\notag\\
&\leq \liminfn \abs{\nabla u-s_nv_n}_{\mb(Q)}+\abs{\nabla^{s_n}(s_nv_0^{s_n})}_\mb+\kappa \norm{\nabla (s_n v_0^{s_n})}_{L^2(Q)} \notag\\
 &\leq \liminfn \, TGV^{1+s_n,\kappa}(u,)\label{l1_mean_proceed}
\end{align}
and the proof is complete for the case $s_n\searrow 0$.\\\\
We next assume that $s_n\nearrow 1$. Consider a $v_0$ such that $T[v_0]=0$, and
\be
\abs{\nabla u- v_0}_{\mathcal{M}_b(Q)}+TV(v_0)+\kappa H^1(v_0)\leq TGV^2(u)+\delta.
\ee
Then, by Proposition \ref{new_equal_r} we have $v_0^n$ such that $v_0^n\to v_0$ strongly in $L^1(Q)$, and 
\be
\limsup_{n\to\infty} TV^{s_n}(v_0^n)\leq TV(v_0)\text{ and }\limsup_{n\to\infty} H^1(v_0^n)\leq H^1(v_0)
\ee
Thus, we have
\begin{align*}
\limsup_{n\to\infty} TGV^{1+s_n,\kappa}(u)
&\leq 
\limsup_{n\to\infty}\abs{\nabla u- s_nv^n_0}_{\mathcal{M}_b(Q)}+s_n [TV^{s_n}(v^n_0)+\kappa H^1(v_0)]\\
&\leq \abs{\nabla u-v_0}_\mb+TV(v_0)+\kappa H^1(v_0)\leq TGV^2(u)+\delta.
\end{align*}
Since $\delta>0$ is arbitrarily, we conclude 
\be
\limsup TGV^{1+s_n,\kappa}(u)\leq TGV^2(u).
\ee 
The $\liminf$ inequality can be achieved in a similar way. Let $v_0^n$ be such that 
\be
\abs{\nabla u- v^n_0}_{\mathcal{M}_b(Q)}+s_nTV^{s_n}(v^n_0)+\kappa H^1(v_0^n)= TGV^{1+s_n,\kappa}(u).
\ee
Hence, up to a subsequence, we have $v_0^n\to \bar v$ and $T[v_0^n]\to T[\bar v]$ strongly in $L^1(Q)$ and $L^1(\partial Q)$ respectively. Therefore, 
\begin{align*}
\liminf_{s_n\nearrow 1} TGV^{1+s_n,\kappa}(u) & \geq \liminfn\,\abs{\nabla u- v^n_0}_{\mathcal{M}_b(Q)}+s_nTV^{s_n}(v^n_0)+\kappa H^1(v_0^n)\\
&\geq \abs{\nabla u- \bar v}_{\mathcal{M}_b(Q)}+TV(\bar v)+\kappa H^1(\bar v)\geq TGV^2(u),
\end{align*}
concluding the proof.
\end{proof}

\begin{define}
\label{def:fractional TGV functional}
Let $r\in\R^+$ and $\alpha\in \R^{\ir+1}_+$ be given. We define the functional $\mathcal I_{\alpha,r}^\kappa :L^1(Q)\to [0,+\infty]$ by
\be
\mathcal I_{\alpha,r}^\kappa(u):=
\begin{cases}
\norm{u-u_\eta}_{L^2(Q)}^2+ TGV^{r,\kappa}_\alpha(u)+\kappa H^2(u)&\text{ if }u\in H_0^2(Q),\\
+\infty &\text{ otherwise. }
\end{cases}
\ee
\end{define}

\begin{theorem}
\label{thm_Gamma_tgv}
Let $\seqn{r_n}\subset [1,T]$ and $\seqn{\alpha_n}\subset \R^{\lfloor r_n\rfloor+1}_+$ be given, satisfying $r_n\to r$ and $\alpha_n\to \alpha$. 
Then the functionals $\mathcal I_{\alpha_n,r_n}^\kappa$ $\Gamma$-converge to $\mathcal I_{\alpha,r}^\kappa$ in the $L^1(Q)$ topology. 
That is, for every $u\in BV(Q)$, the following two conditions hold.\\\\
{\rm (Liminf inequality)} If 
\be u_n\to u\text{ in }L^1(Q)\text{ for }u\in H_0^2(Q)\ee
then 
\be
\mathcal I_{\alpha,r}^\kappa(u)\leq \liminf_{n\to +\infty}\mathcal I_{\alpha_n,r_n}^\kappa(u_n).\ee
{\rm (Recovery sequence)} For each $u\in H_0^2(Q)$, there exists $\seqn{u_n}\subset H_0^2(Q)$ such that 
\be u_n\to u\text{ in }L^1(Q)\ee
and 
\be
\limsup_{n\to +\infty}\mathcal I_{\alpha_n,r_n}^\kappa(u_n)\leq \mathcal I_{\alpha,r}^\kappa(u).
\ee
\end{theorem}
We split the proof of Theorem \ref{thm:new-Gamma thm} into two propositions.

\begin{proposition}
Let $\seqn{r_n}\subset [1,T]$ and $\seqn{\alpha_n}\subset \R^{\lfloor r_n\rfloor+1}_+$ be given, satisfying $r_n\to r$ and $\alpha_n\to \alpha\in\R^{\ir+1}_+$. 
For every $n\in \N$, let $u_n\in H_0^2(Q)$ be such that 
\be
\sup\flp{\mathcal I_{\alpha_n,r_n}^\kappa(u_n):\,\,n\in\N}<+\infty.
\ee
Then there exists $u\in H_0^2(Q)$ such that, up to a (not relabeled) subsequence,
\be\label{weak_BV_liminf}
u_n\wto  u\text{ weakly}\text{ in }H_0^2(Q)
\ee
and 
\be\label{weak_BV_liminf1}
\liminf_{n\to\infty}{TGV_{\alpha_n}^{r_n,\kappa}(u_n)}\geq {TGV_{\alpha}^{r,\kappa}(u)},
\ee
with
\be\label{weak_BV_liminf2}
\liminfn \,\mathcal I_{\alpha_n,r_n}^\kappa(u_n)\geq \mathcal I_{\alpha,r}^\kappa(u).
\ee
\end{proposition}
\begin{proof}
We again consider only the case $\seqn{r_n}\subset[1,2]$, and the case $\seqn{r_n}\subset [2,T]$ (when $T>2$) can be dealt with analogously.\\\\
In view of the Huber-regularization $\kappa H^2(u)$, we have \eqref{weak_BV_liminf} immediately. Next, write $r_n=1+s_n$, $s_n\in[0,1]$.
Assume $s_n\to s\in(0,1]$ first. In this case, Proposition \ref{assum_basic_B} holds for $TV^{s_n}$, and by \cite[Proposition 4.4]{liu2016Image}, we infer
\eqref{weak_BV_liminf1} and \eqref{weak_BV_liminf2}.\\\\
Now assume $s_n\searrow 0$. The proof is similar to that from Proposition \ref{thm_asymptitic_PGV}. Without loss of generality, 
we can assume that $\alpha_n=1$, and let $\seqn{v_n}\subset L^1(Q;\rn)$ be such that 
\be
TGV^{1+s_n,\kappa}(u_n)=\abs{\nabla u_n- s_nv^n_0}_{\mathcal{M}_b(Q)}+s_nTV^{s_n}(v^n_0)+\kappa H^1(v_0^n).
\ee
Then, we have, for sufficiently large $n\in\N$, 
\be
\abs{s_nv^n_0}_{\mathcal{M}_b(Q)}+s_nTV^{s_n}(v^n_0)+\kappa H^1(v_0^n)\leq TV(u)+1<+\infty.
\ee
Then, by the same computations from \eqref{l1_mean_proceed}, and by \eqref{weak_BV_liminf}, we conclude that 
\be
\liminfn\, TGV^{1+s_n}(u_n)\geq TV(u),
\ee
and hence \eqref{weak_BV_liminf2}, as desired.
\end{proof}

\begin{proposition}\label{new_equal}
Let $\seqn{r_n}\subset [1,T]$ and $\seqn{\alpha_n}\subset \R^{\lfloor r_n\rfloor+1}_+$ be given, satisfying $r_n\to r$ and $\alpha_n\to \alpha\in\R^{\ir+1}_+$. 
Then for every $u\in H_0^2(Q)$ there exists $\seqn{u_n}\subset H_0^2(Q)$ such that $u_n\wto u$ in $H^{-1}(Q)$, and
\be
\limsup_{n\to\infty}{\TGV_{\alpha_n}^{r_n,\kappa}(u_n)}\leq {\TGV_\alpha^{r,\kappa}(u)}.
\ee
\end{proposition}
\begin{proof}
This is a direct consequence of Proposition \ref{thm_asymptitic_PGV}， by choosing $u_n=u$.
\end{proof}

\begin{proposition}
Let $u_\eta\in L^2(Q)$. Let $r\in[1,T]$, and let $\alpha\in\R^{\ir+1}$. Then, there exists a unique $u_{\alpha,r}\in H_0^2(Q)$ such that
\be
u_{\alpha,r}\in\argmin\flp{\norm{u-u_\eta}^2_{L^2(Q)}+\TGV_\alpha^{r,\kappa}(u)+\kappa H^2(u):\, u\in H_0^2(Q)}.
\ee
\end{proposition}
\begin{proof}
The proof can be directly concluded from \cite[Proposition 5.3]{liu2016Image} for $r\in\R^+\setminus\N$. The case that $r\in\N$ can be obtained from the standard $TGV$ result.
\end{proof}

 We define the symmetric derivative by
\be
\mathcal E^sv:=\frac12\fmp{\nabla^s v+(\nabla^sv)^T}.
\ee
\begin{proposition}\label{prop_assum_symderi}
We collect several properties of seminorm $\abs{\FE^sv}_{\mb(Q)}$, with function $v\in BV^s(Q;\rn)\cap H_0^1(Q;\rn)$ with zero boundary conditions.
\begin{enumerate}[1.]
\item
The null space of seminorm $\abs{\FE^sv}_{\mb(Q)}+\kappa H^1(v)$ has finite dimension.
\item
For every $v\in BV^s(Q;\rn)\cap H^1_0(Q;\rn)$ there exists $\seqn{u_n}\subset C^\infty(\bar Q,\rn)$ such that 
\be
u_n\to u\text{ strongly in }L^1(Q;\rn)\text{ and }\abs{\FE^sv_n}_{\mb(Q)}+H^1(v_n)\to \abs{\FE^sv}_{\mb(Q)}+H^1(v);
\ee
\item
Let $\seqn{v_n}\subset H_0^1(Q;\rn)$ and $\seqn{s_n}\subset [\sigma,1-\sigma]$ be such that 
\be 
\sup\flp{\norm{v_n}_{L^1(Q;\rn)}+\abs{\FE^sv_n}_{\mb(Q)}+\kappa H^1(v_n):\,\, n\in\N}<+\infty.
\ee
Then there exist $v\in BV^s(Q;\R^N)\cap H_0^1(Q;\rn)$ such that, up to a subsequence (not relabeled), 
\be 
v_n\to v\text{ strongly in }L^1(Q;\R^N),
\ee
and 
\be\label{eq_sd_uniform_bdd2}
\liminfn \,\abs{\FE^sv_n}_{\mb(Q)}+\kappa H^1(v_n)\geq \abs{\FE^sv}_{\mb(Q)}+ \kappa H^1(v).
\ee
\end{enumerate}
\end{proposition}
\begin{proof}
We prove the lower semi-continuity. Since $v_n=(v_n^1,v_n^2)\in H_0^1(Q;\rn)$, $\nabla v$ is defined a.e., and we write $\nabla v = [\partial_1 v_1, \partial_2 v_1; \partial_1 v_2,\partial_2v_2]$. By the zero boundary condition, we have 
\be
\partial_i v_{j,n}\to \partial_i v_j\text{ a.e. for }i,j\in\flp{1,2}.
\ee
Thus, by Fatou's lemma, we conclude that 
\be
\liminfn \,\abs{\FE^sv_n}_{\mb(Q)}\geq  \abs{\FE^sv}_{\mb(Q)},
\ee
and hence \eqref{eq_sd_uniform_bdd2}, as desired.
\end{proof}
Thus, Proposition \ref{prop_assum_symderi} shows that the collection
\be
\Pi:=\flp{\abs{\FE^s v}_{\mb(Q)}+\kappa H^1(v):\,\, 0\leq s\leq 1},
\ee
satisfies \cite[Assumptions 3.2 \& 3.3]{liu2016Image}, and also the argument used in Proposition \ref{thm_asymptitic_PGV}. Therefore, the functional 
\be
\mathcal I_{\alpha,r}^\kappa(u):=
\begin{cases}
\norm{u-u_\eta}_{L^2(Q)}^2+ TGV^{r,\kappa}_\alpha(u)+\kappa H^2(u)&\text{ if }u\in H_0^2(Q),\\
+\infty &\text{ otherwise. }
\end{cases}
\ee
satisfies  the $\Gamma$-convergence results from Theorem \ref{thm_Gamma_tgv}.

\begin{remark}
We shall only work in the $NsGTV$ framework from now, as we know from Proposition \ref{prop_assum_symderi} that $\abs{\FE^s v}_{\mb(Q)}$ and $\abs{\nabla^sv}_{\mb(Q)}$ 
behave quite similarly, due to the presence of the Huber-regularization term $\kappa H^1(v)$.
\end{remark}

\subsubsection{The real $r$-order $NsTGV^r$ seminorm with underlying Euclidean norm}\label{sec_tgv_rp}

We define the (finite) mass of a vector-valued measure $\mu$: $Q\to\rn$, with underlying Euclidean $\ell^p$ norm, by
\be
\abs{\mu}_{\mathcal M_{\ell^p}(Q;\rn)}:=\sup\flp{\int_Q\vp\,d\mu:\,\, \vp\in C_c^\infty(Q;\rn),\,\,\abs{\vp}_{\ell^{p^\ast}}\leq 1}.
\ee

\begin{define}
Let $r\in\R^+$ be given, and write $r=k+s$, $k=\ir$, and let $\alpha=(\alpha_0,\alpha_1,\alpha_2,\ldots, \alpha_{k})\in\R_+^{k+1}$, and $p=(p_0,p_1,\ldots, p_{k})\in[1,+\infty]^{k+1}$. 
For every $u\in L^1(Q)$, we define its fractional $TGV^{k+s}$ seminorm as follows.
\begin{enumerate}[C{a}se 1.] 
\item
For $k=1$, thus $\alpha=(\alpha_0,\alpha_1)\in\R_+^2$ and $p=(p_0,p_2)\in[1,+\infty]^2$, we set
\be
NsTGV_{\alpha,p}^{1+s,\kappa}(u):=\inf\flp{\alpha_0\abs{\nabla u- sv_0}_{\mathcal{M}_{{p_0}}(Q;\rn)}+\alpha_1s [TV_{\ell^{p_1}}^s(v_0)+\kappa H^1(v_0)]:\,v_0\in  H_0^1(Q;\rn)}.
\ee
\item
For $k>1$, we set
\begin{align*}
NsTGV_{\alpha,p}^{k+s,\kappa}(u)&:=\inf\flp{\alpha_0\abs{\nabla  u- v_0}_{\mathcal{M}_{\ell^{p_0}}(Q)}+\alpha_1\abs{\nabla  v_0-v_1}_{{\mathcal{M}_{\ell^{p_1}}(Q)}}+\right.\notag\\
&\left. \cdots+\alpha_{k-1}\abs{\nabla v_{k-2}-sv_{k-1}}_{{\mathcal{M}_{\ell^{p_{k-1}}}(Q)}}+\alpha_{k}s[TV^s_{\ell^{p_k}}(v_{k-1})+\kappa H^1(v_{k-1})],\right.\notag\\
&\left. v_i\in BV(Q;\mathbb M^{N^{l}})\text{ for }0\leq i\leq k-2,\,\,v_{k-1}\in H_0^1(Q;\mathbb M^{N^\ir})}.
\end{align*}
\end{enumerate}
\end{define}

\subsection{Infimal convolution $\RVL$}\label{sec_RVL_final}

We finally arrive at our proposed regularizer, which unifies all the regularizers introduced above. Recall we write $r_2 = k+s$, where $k\in\N$ and $s\in[0,1)$.
\begin{define}[The $\RVL^r$ seminorms]\label{TVL_fractional}
Let $r=(r_1,r_2)\in\R_+^2$ and let $\alpha=(\alpha_0,\alpha_1,\alpha_2,\ldots, \alpha_{k})\in\R_+^{k+1}$ and $p=(p_0,p_1,\ldots, p_{k})\in[1,+\infty]^{k+1}$.
For every $u\in L^1(Q)$, we define the unified $\RVL^{r,\kappa}_{\alpha,p}$ seminorm as follows.
\begin{enumerate}[C{a}se 1.] 
\item
For $k=1$, thus $\alpha=(\alpha_0,\alpha_1)\in\R_+^2$ and $p=(p_0,p_1)\in[1,+\infty]^{2}$,
\be
\operatorname{RVL}_{\alpha,p}^{r,\kappa}(u):=\inf\flp{\alpha_0\abs{\nabla^{r_1} u- sv_0}_{\mathcal{M}_{\ell^{p_0}}(Q)}+\alpha_1s [\operatorname{TV}^s_{\ell^{p_1}}(v_0)+\kappa H^1(v_0)]:\,v_0\in  H_0^1(Q;\rn)}.
\ee
\item
For $k>1$,
\begin{align}
\operatorname{RVL}^{r,\kappa}_{\alpha,p}(u):=\inf\flp{\alpha_0\abs{\nabla^{r_1}  u- v_0}_{\mathcal{M}_{\ell^{p_0}}(Q)}+\alpha_1\abs{\nabla  v_0-v_1}_{{\mathcal{M}_{\ell^{p_1}}(Q)}}+\right.\notag\\
\left. \cdots+\alpha_{k-1}\abs{\nabla v_{k-2}-v_{k-1}}_{{\mathcal{M}_{\ell^{p_{k-1}}}(Q)}}+\alpha_{k}[TV_{\ell^{p_k}}^s(v_{k-1})+\kappa H^1(v_{k-1})],\right.\notag\\
\left. v_i\in BV(Q;\mathbb M^{N^{l}})\text{ for }0\leq i\leq k-2,\,\,v_{k-1}\in H_0^1(Q;\mathbb M^{N^k})}.
\notag
\end{align}
\end{enumerate}
\end{define}

Then, for given $r=(r_1,r_2)\in\R_+^2$, $\alpha\in \R^{\lfloor r_2\rfloor+1}_+$, and $p\in [1,+\infty]^{\lfloor r_2\rfloor+1}$, we define the functional $\mathcal I_{\alpha,r,p}^\kappa$: $L^1(Q)\to [0,+\infty]$ by
\be
\mathcal I_{\alpha,r,p}^\kappa(u):=
\begin{cases}
\norm{u-u_\eta}_{L^2(Q)}^2+ \RVL^{r,\kappa}_{\alpha,p}(u)+\kappa H^{\lfloor r_1\rfloor+1}(u)&\text{ if }u\in H_0^{\lfloor r_1\rfloor+1}(Q),\\
+\infty &\text{ otherwise. }
\end{cases}
\ee

\begin{remark}\label{primal_auxiliary}
For parameter $r=(r_1,r_2)$ which controls the order of $\RVL$ regularizer, we call $r_1$ the primal order (as it directly works on $u$)  and $r_2$ the auxiliary order (as it controls how we defines the order of auxiliary variable $v$).
\end{remark}
\begin{theorem}
\label{thm:new-Gamma thm}
Let $\seqn{r_n}=\seqn{(r_{1,n},r_{2,n})}\subset [1,T]^2$, $\seqn{\alpha_n}\subset \R^{\lfloor r_{2,n}\rfloor+1}_+$, and $\seqn{p_n}\subset [1,+\infty]^{\lfloor r_{2,n}\rfloor+1}$ be given, satisfying $r_n\to r$, $\alpha_n\to \alpha$, and $p_n\to p$. 
Then the functionals $\mathcal I_{\alpha_n,r_n,p_n}^\kappa$ $\Gamma$-converge to $\mathcal I_{\alpha,r,p}^\kappa$ in the $L^1(Q)$ topology. That is, 
for every $u\in BV(Q)$ the following two conditions hold.\\\\
{\rm (Liminf inequality)} If 
\be u_n\to u\text{ in }L^1(Q)\text{ for }u\in H_0^{\lfloor r_1\rfloor+1}(Q)\ee
then 
\be
\mathcal I_{\alpha,r,p}^\kappa(u)\leq \liminf_{n\to +\infty}\mathcal I_{\alpha_n,r_n,p_n}^\kappa(u_n).\ee
{\rm (Recovery sequence)} For each $u\in H_0^{\lfloor r_1\rfloor+1}(Q)$, there exists
a sequence $\seqn{u_n}\subset H_0^{\lfloor r_1\rfloor+1}(Q)$ such that 
\be u_n\to u\text{ in }L^1(Q)\ee
and 
\be
\limsup_{n\to +\infty}\,\mathcal I_{\alpha_n,r_n,p_n}^\kappa(u_n)\leq \mathcal I_{\alpha,r,p}^\kappa(u).
\ee
\end{theorem}

\begin{proof}

We consider only the case $1\leq r_1\leq 2$. The other cases are rather similar. Also, in view of Remark \ref{an_iso_equ}, and the arguments in the proof of Proposition \ref{thm:new-Gamma}, we could assume for simplicity that $p_n=p=(2,2,2\ldots,2)$. That is, we are operating under the standard Euclidean norm, and we abbreviate $\RVL_{\alpha,p}^{r,\kappa}$ by $\RVL_\alpha^{r,\kappa}$. \\\\
Next, write $r_1=k_1+s$, $k_1\in\N$, $s\in[0,1)$. Let $u_n$ be a sequence weakly converging to $u$ in $H^2(Q)$, and we claim that 
\be\label{snunsnu_eq3}
\text{$\nabla^{s_n}u_n\to \nabla^s u$ strongly in $L^1(Q)$.}
\ee
We consider first that $s_n\to s\in(0,1]$. Let $x_2\in I=(0,1)$ be fixed and define $w_n(t) := u_n(t,x_2)$. Without loss
of generality, we have $w_n(t)\in H_0^2(I)$ and $w_n\wto w$ weakly in $H^2(I)$, where $w(t):=u(t,x_2)\in H_0^2(I)$.\\\\
Since $\seqn{w_n}\subset H_0^2(I)$, we have $\seqn{w_n}\subset \mathbb I^{1}(L^1(I))$, and
\begin{align*}
\norm{d^{s_n}w_n-d^sw}_{L^1(I)}&\leq \norm{d^{s_n}w_n-d^{s_n}w}_{L^1(I)}+\norm{d^{s_n}w-d^sw}_{L^1(I)}\\
&\leq \norm{d^{s_n}(w_n-w)}_{L^1(I)}+\norm{(d^{s_n}-d^s)w}_{L^1(I)}.
\end{align*}
Since $w\in \mathbb I^2(L^1(I))$, there exists $f\in L^1(I)$ such that $w=\mathbb I^2[f]$, and 
\begin{align*}
\norm{(d^{s_n}-d^s)w}_{L^1(I)}& = \norm{(\mathbb I^{-s_n}-\mathbb I^{-s})\mathbb I^2f}_{L^1(I)} \\
&= \norm{(\mathbb I^{2-s_n}-\mathbb I^{2-s})[f]}_{L^1(I)} \leq \abs{\mathbb I^{2-s_n}-\mathbb I^{2-s}}\norm{f}_{L^1(I)}\\
&=\abs{\mathbb I^{2-s_n}-\mathbb I^{2-s}}\norm{d^2 w}_{L^1(I)}.
\end{align*}
On the other hand, we have 
\begin{align*}
 \norm{d^{s_n}(w_n-w)}_{L^1(I)} &=  \norm{\mathbb I^{1-s_n}(\mathbb I^{1}[f_n]-\mathbb I^{1}[f])}_{L^1(I)}\\
 &\leq  \abs{\mathbb I^{1-s_n}}\norm{\mathbb I^{1}[f_n]-\mathbb I^{1}[f]}_{L^1(I)}= \abs{\mathbb I^{1-s_n}}\norm{dw_n-dw}_{L^1(I)}.
\end{align*}
Thus, we have 
\begin{align}
&\norm{\nabla^{s_n}u_n-\nabla^s u}_{L^1_{\ell^2}(Q;\rn)}\\
&\leq \norm{\nabla^{s_n}u_n-\nabla^s u}_{L^1_{\ell^1}(Q;\rn)}\notag\\
&\leq  \norm{\nabla^{s_n}u_n-\nabla^{s_n} u}_{L^1_{\ell^1}(Q;\rn)}+ \norm{\nabla^{s_n}u-\nabla^s u}_{L^1_{\ell^1}(Q;\rn)}.
\label{snunsnu_eq2}
\end{align}
We observe that 
\begin{align*}
 \norm{\nabla^{s_n}u_n-\nabla^{s_n} u}_{L^1_{\ell^1}(Q;\rn)}
&=\int_Q\abs{\partial_1^{s_n}u_n-\partial_1^{s_n} u}dx+\int_Q\abs{\partial_2^{s_n}u_n-\partial_2^{s_n} u}dx\\
&\leq \abs{\mathbb I^{1-s_n}}\fmp{\int_Q\abs{\partial_1^1u_n-\partial_1^1 u}dx+\int_Q\abs{\partial_2^1u_n-\partial_2^1 u}dx}\\
&\leq \abs{\mathbb I^{1-s_n}}\norm{\nabla u_n-\nabla u}_{L^2(Q)}.
\end{align*}
From \cite[Proposition 4.6]{liulud2019Image}, we have
\be
\sup\flp{\abs{\mathbb I^{1-t}}:\,\, t\in[0,1]}\leq\sup_{t\in[0,1]} \frac{1}{\Gamma(1-t+1)}<+\infty,
\ee
and hence
\begin{align}
\limsup_{s_n\to s}\norm{\nabla^{s_n}u_n-\nabla^{s_n} u}_{L^1(Q;\rn)}\leq \limsup_{s_n\to s}\norm{\nabla^{s_n}u_n-\nabla^{s_n} u}_{L^1_{\ell^1}(Q;\rn)}\notag\\
\leq \sup\flp{\abs{\mathbb I^{1-t}}:\,\, t\in[0,1]}\limsup_{n\to\infty}\norm{\nabla^{s_n}u_n-\nabla^{s_n} u}_{L^2(Q;\rn)}=0.\label{snunsnu_eq}
\end{align}
On the other hand
\begin{align*}
\norm{\nabla^{s_n}u-\nabla^s u}_{L^1_{\ell^1}(Q;\rn)}
&=\int_Q\abs{\partial_1^{s_n}u-\partial_1^s u}dx+\int_Q\abs{\partial_2^{s_n}u-\partial_2^s u}dx\\
&\leq \abs{\mathbb I^{2-s_n}-\mathbb I^{2-s}}\fmp{\int_Q\abs{\partial_1^2 u}dx+\int_Q\abs{\partial_2^2 u}dx}\\
&\leq \abs{\mathbb I^{2-s_n}-\mathbb I^{2-s}}\norm{u}_{H^2(Q)}.
\end{align*}
Thus, we have 
\begin{align*}
\limsup_{n\to\infty}\norm{\nabla^{s_n}u-\nabla^s u}_{L^1(Q;\rn)}&\leq \limsup_{n\to\infty}\norm{\nabla^{s_n}u-\nabla^s u}_{L^1_{\ell^1}(Q;\rn)}\\
&\leq \norm{u}_{H^2(Q)} \limsup_{s_n\to s}\abs{\mathbb I^{2-s_n}-\mathbb I^{2-s}}=0,
\end{align*}
as the operator $\mathbb I^r$ is strictly continuous for $r>0$. 
This, together with \eqref{snunsnu_eq} and \eqref{snunsnu_eq2}, gives \eqref{snunsnu_eq3}. Now assume that $s_n\to 0$. 
In this case we write $u=(u/2,u/2)$ in the above arguments, which then gives $\nabla^{s_n}u\to u$ strongly in $L^1(Q)$, as desired.\\\\
This, combined with Theorem \ref{thm_Gamma_tgv}, concludes the proof.
\end{proof}

Finally, we introduce the training scheme with training ground
\be\label{box_const}
\T:=[0,T]^{1+T}\times[1,T]^2\times [1,+\infty]^{1+T},
\ee
where $T\in\N$ is a given box-constraint. As introduced in Section \ref{sec:intro}, our semi-supervised (bilevel) training scheme $\mathcal T$ can be written as 
\begin{flalign}
\text{Level 1. }&\,\,\,\,\,\,\,\,\,\,\,\,\,\,\,\,(\alpha_{\T}, r_{\T},p_\T)\in\mathbb A[\T]:=\argmin\flp{\norm{u_{\alpha,r,p}-u_c}_{L^2(Q)}^2:\,\,(\alpha,r,p)\in\mathbb T}\tag{$\mathcal T$-L1}\label{S_scheme_L1_D0}&\\
\text{Level 2. }&\,\,\,\,\,\,\,\,\,\,\,\,\,\,\,\,u_{\alpha,r,p}:=\argmin\flp{\norm{u-u_\eta}_{L^2(Q)}^2+ \RVL^{r,\kappa}_{\alpha,p}(u)+\kappa H^{\lfloor r_1\rfloor+1}(u):\right.\\
&\,\,\,\,\,\,\,\,\,\,\,\,\,\,\,\,\,\,\,\,\,\,\,\,\,\,\,\,\,\,\,\,\,\,\,\,\,\,\,\,\,\,\,\,\,\,\,\,\,\,\,\,\,\,\,\,\,\,\,\,\,\,\,\,\,\,\,\,\,\,\,\,\,\,\,\,\,\,\,\,\,\,\,\,\,\,\,\,\,\,\,\,\,\,\,\,\,\,\,\,\,\,\,\,\,\,\,\,\,\,\,\,\,\,\,\,\,\,\,\,\,\,\,\,\,\,\,\,\,\,\,\,\,\,\,\,\,\,\,\,\left. u\in H_0^{\lfloor r_1\rfloor+1}(Q)}.\tag{$\mathcal T$-L2}\label{S_scheme_L2_D0}&
\end{flalign}

\begin{theorem}\label{thm_TS_S_solution}
Let a Training Ground $\mathbb T$ be given. 
Then, the training scheme $\mathcal T$ (\eqref{S_scheme_L1_D0}-\eqref{S_scheme_L2_D0}) 
admits at least one solution $( \alpha_\T, r_\T,p_\T)\in \T$,
and provides a corresponding optimal reconstructed image $u_{\alpha_\T, r_\T,p_\T}$.
\end{theorem}

\begin{proof}[Proof of Theorem \ref{thm_TS_S_solution}]
Let the \emph{TrainingGround} $\mathbb T$ be fixed. Let $\seqn{\alpha_n,r_n}\subset  \mathbb T$ be a minimizing sequence 
obtained from \eqref{S_scheme_L1_D0}. Then, since $\mathbb T$ is compact, up to a subsequence (not relabeled), 
there exists $(\bar\alpha,\bar  r, \bar p)\in\mathbb T$ such that $(\alpha_n,r_n,p_n)\to(\bar \alpha,\bar r,\bar p)\in\mathbb T$ and 
\be\label{goes_to_hell_heaven}
\limn \,\norm{u_c-u_{\alpha_n,r_n,p_n}}_{L^2(Q)}^2\to m:=\inf\flp{\norm{u_c-u_{\alpha,r,p}}_{L^2(Q)}^2:\,\,(\alpha,r)\in\mathbb T}.
\ee
We show that $(\bar\alpha,\bar r,\bar p)\in\mathbb A[\mathbb T]$ (defined in \eqref{S_scheme_L1_D0}), and we split our arguments into two cases.\\\\
\underline{Case 1:} Assume $\bar \alpha>0$. That is, every components of $\bar \alpha$ is non-zero. Then, in view of Theorem \ref{thm:new-Gamma thm} and the general properties of the $\Gamma$-convergence, we have 
\be
u_{\alpha_n,r_n,p_n}\wto u_{\bar \alpha,\bar r,\bar p}\text{ weakly}\text{ in }L^2(Q).
\ee
Therefore, we have
\be
\norm{u_{\bar\alpha,\bar r,\bar p}-u_c}_{L^2(Q)}\leq \liminfn \norm{u_{\alpha_n,r_n,p_n}-u_c}_{L^2(Q)}=m,
\ee
which implies that $(\alpha_n,r_n,p_n)\in\mathbb A[\mathbb T]$, completing the proof.\\\\
\underline{Case 2:} Assume now that at least one component of $\bar\alpha$ is zero. In this case, by \eqref{goes_to_hell_heaven}, 
there exists $\bar u\in L^2(Q)$ such that, up to a subsequence, $u_{\alpha_n,r_n,p_n}\wto \bar u$ in $L^2(Q)$. We claim that $u_{\alpha_n,r_n,p_n}\to u_\eta$ strongly in $L^2(Q)$. Extend $u_\eta$ to
zero outside $Q$, and define 
\be
u_{\eta}^\e:=u_\eta\ast \eta_\e
\ee
where $\eta_\e$ is some mollifier, whose particular form is however, not very relevant. Then we have $u_\eta^\e\in C_c^\infty(\rn)$, and $u_{\eta}^\e\to u_\eta$ strongly in $L^2(\rn)$. \\\\
We only consider the case that $\seqn{r_{2,n}}\subset [1,2]$, as the other cases can be handled similarly. That is, we have $\seqn{\alpha_n} =\seqn{(\alpha_{1,n},\alpha_{2,n})}\subset [0,T]^2$ and $\bar\alpha = (\bar\alpha_1,\bar\alpha_2)\in[0,T]^2$.  Assume first that $\bar \alpha_1=0$, then by the optimality of \eqref{S_scheme_L2_D0}, we have
\begin{align*}
&\norm{u_{\alpha_n,r_n,p_n}-u_\eta}^2_{L^2(Q)}+\RVL_{\alpha_n,p_n}^{r_n,\kappa}(u_{\alpha_n,r_n,p_n})\\
&\leq \norm{u_{\eta}^\e-u_\eta}^2_{L^2(Q)}+\RVL_{\alpha_n,p_n}^{r_n,\kappa}(u_\eta^\e)\leq \norm{u_{\eta}^\e-u_\eta}^2_{L^2(Q)}+\alpha_{1,n}TV^{r_{1,n}}(u_\eta^\e)\\
&\leq \norm{u_{\eta}^\e-u_\eta}^2_{L^2(Q)}+\alpha_{1,n} \norm{u_\eta^\e}_{W^{\lfloor r_1\rfloor+1,+\infty}(\rn)}.
\end{align*}
That is, 
\be
\norm{u_{\alpha_n,r_n,p_n}-u_\eta}^2_{L^2(Q)}\leq  \norm{u_{\eta}^\e-u_\eta}^2_{L^2(Q)}+\alpha_{1,n} \norm{u_\eta^\e}_{W^{\lfloor r_1\rfloor+1,+\infty}(\rn)},
\ee
and we conclude by first taking the limit $\alpha_{1,n}\to \bar\alpha_1=0$, and then the limit $\e\to 0$.\\\\
Now we assume $\bar\alpha_2=0$. We again observe that 
\begin{align*}
&\norm{u_{\alpha_n,r_n,p_n}-u_\eta}^2_{L^2(Q)}+\RVL_{\alpha_n,p_n}^{r_n,\kappa}(u_{\alpha_n,r_n,p_n})\\
&\leq \norm{u_{\eta}^\e-u_\eta}^2_{L^2(Q)}+\RVL_{\alpha_n,p_n}^{r_n,\kappa}(u_\eta^\e)\leq \norm{u_{\eta}^\e-u_\eta}^2_{L^2(Q)}+\alpha_{2,n}TV^{r_{1,n}+r_{2,n}}(u_\eta^\e)\\
&\leq \norm{u_{\eta}^\e-u_\eta}^2_{L^2(Q)}+\alpha_{2,n} \norm{u_\eta^\e}_{W^{\lfloor r_1\rfloor+\lfloor r_2\rfloor+2,+\infty}(\rn)},
\end{align*}
and we conclude again by first taking the limit $\alpha_{1,n}\to \bar\alpha_1=0$, and then the limit $\e\to 0$.
\end{proof}

\begin{remark}\label{pract_box_c}
Note that the box constraint, defined in \eqref{box_const}, is only used to guarantee that a minimizing sequence $\seqn{(\alpha_n,r_n,p_n)}$, obtained from \eqref{S_scheme_L1_D0}, has a convergent subsequence. Alternatively, different box-constraints for different parameter $\alpha$, $r$, and $p$ might be enforced, such as 
\begin{align}\label{un_boxed}
(r,\alpha,p)=&(r_1,r_2,\alpha_0,\ldots,\alpha_{\lfloor T_2\rfloor+1},p_0,\ldots,p_{\lfloor T_2\rfloor+1})\\&\in[1,T_1]\times [1,T_2]\times [0,P_1]\times\cdots\times[0,P_i]\times\cdots\times [0,P_{\lfloor T_2\rfloor+2}]\times [1,+\infty]^{\lfloor T_2\rfloor+2},
\end{align}
where $1<T_1$, $T_2<+\infty$, and $0<P_i<+\infty$, for each $i\in\N$.\\\\
Note that in \eqref{un_boxed} we have the auxiliary order $r_2$ of $\RVL$ regularizer belongs to $[1,T_2]$, and hence the number of parameter $\alpha$ and $p$ then determined by the integer part of $T_2$. 
\end{remark}

\section{Simulations and insights}\label{num_sim_pd}
In this Section we perform numerical simulations of the bilevel scheme $\mathcal T$ using the corrupted image $u_\eta$ and the corresponding clean image $u_c$ shown in Figure \ref{fig:clean_noise}. The Level 2 problem \eqref{S_scheme_L2_D0} is solved via the primal-dual algorithm studied in \cite{chambolle2011first, chambolle2017stochastic}.\\\\
To make an appropriate comparison, we apply our proposed training scheme $\mathcal T$ (\eqref{S_scheme_L1_D0}-\eqref{S_scheme_L2_D0}) 
on the training data $(u_\eta,u_c)$ shown in Figure \ref{fig:clean_noise} with the following different training grounds (recall Remark \ref{pract_box_c}):
\be\label{test_ground0}
\T_0=\flp{(r,\alpha,p)}:= \flp{1}^2\times [0,1]\times[1,+\infty],
\ee
\be\label{test_ground1}
\T_1=\flp{(r,\alpha,p)}:=\flp{1}\times[1,2]\times\fmp{0,1}^2\times \flp{1}^2,
\ee
\be\label{test_ground2}
\T_2=\flp{(r,\alpha,p)}:=\fmp{1,2}^2\times [0,1]^2\times [1,+\infty]^2.
\ee
Note that the training ground $\T_0$ gives the original training scheme $\mathcal B$ (\eqref{intro_B_train_level1}-\eqref{intro_B_train_level2}) with $TV$ regularizer only. In
the training ground $\T_1$ in \eqref{test_ground1}, the auxiliary order $r_2$ of $\RVL$ regularizer, defined in \eqref{TVL_fractional}, can vary inside interval $[1,2]$. That is, the training ground $\T_1$ allows the $\RVL$ regularizer to provide a unified approach to the classical regularizer $TV$ and $TGV^2$. The training ground $\T_2$ provides an even further extension compared to $\T_1$, by allowing the primal order $r_1$ to vary in $[1,2]$.

\begin{figure}[!h]
  \centering
        \includegraphics[width=1.0\linewidth]{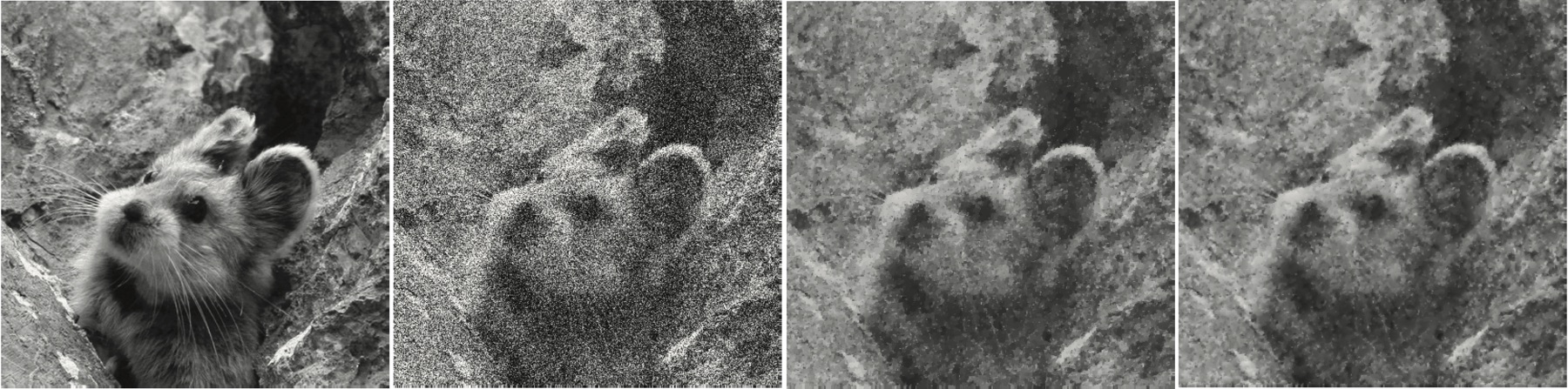}
\caption{Left to right: clean test image; corrupted (noisy) version (with artificial Gaussian noise); optimal reconstructed image provided by Scheme $\mathcal B$; optimal reconstructed image provided by Scheme $\mathcal T$ with training ground $\T_2$ defined in \eqref{test_ground2}.}
\label{fig:clean_noise}
\end{figure}

We summarize our simulation results in Table \ref{table_test_result} below. Recall the concept of \emph{Minimum Assessment Value} (MAV) used in \eqref{mav_value}.

\begin{table}[!h]
\begin{tabular}{|l|l|l|l|l|l|}
\hline
TG & optimal solution & MAV   \\ \hline
$\T_0$ & $\alpha_{\T_0}=0.056$  & 17.482    \\ \hline
$\T_1$ &$\alpha_{\T_1}=(0.052,0.798)$, $r_{\T_1}=(1,1.2)$  &17.014  \\ \hline
$\T_2$ &$\alpha_{\T_2}=(0.037,0.911)$, $r_{\T_2}=(1.9,1.2)$, $p_{\T_2}=(1.1,3.5)$  &16.214  \\ \hline
\end{tabular}
\caption{Minimum assessment value (MAV) for scheme $\mathcal T$ over the training ground (TG) defined in \eqref{test_ground0}, \eqref{test_ground1}, and \eqref{test_ground2}.} 
\label{table_test_result}
\end{table}

We see from Table \ref{table_test_result} that as we expand the training ground $\T$, the MSV value starts dropping. That is, the scheme $\mathcal T$ (\eqref{S_scheme_L1_D0}-\eqref{S_scheme_L2_D0})  with regularizer $\RVL$ does indeed provide a better solution compared to the scheme $\mathcal B$ (\eqref{intro_B_train_level1}-\eqref{intro_B_train_level2}).\\\\
Finally, in order to gain further insights on the numerical landscape of assessment function defined as
\be
\mathcal A(\alpha,r,p):=\norm{u_{\alpha,r,p}-u_c}_{L^2(Q)},
\ee
we compute it for a training ground $\T_3:=[1,2]\times\flp{1}\times [0,1]\times\flp{1}$. That is, in $\T_3$ we only allow the primal order $r_1$ and the associated intensity parameter $\alpha_1$ to change, and freeze all other parameters. Then, the numerical landscapes of the assessment function $\mathcal A(\alpha,r)$ are visualized in Figure \ref{fig:2nd_counter_exam}. 

\begin{figure}[!h]
\begin{subfigure}{.495\textwidth}
  \centering
        \includegraphics[width=1.0\linewidth]{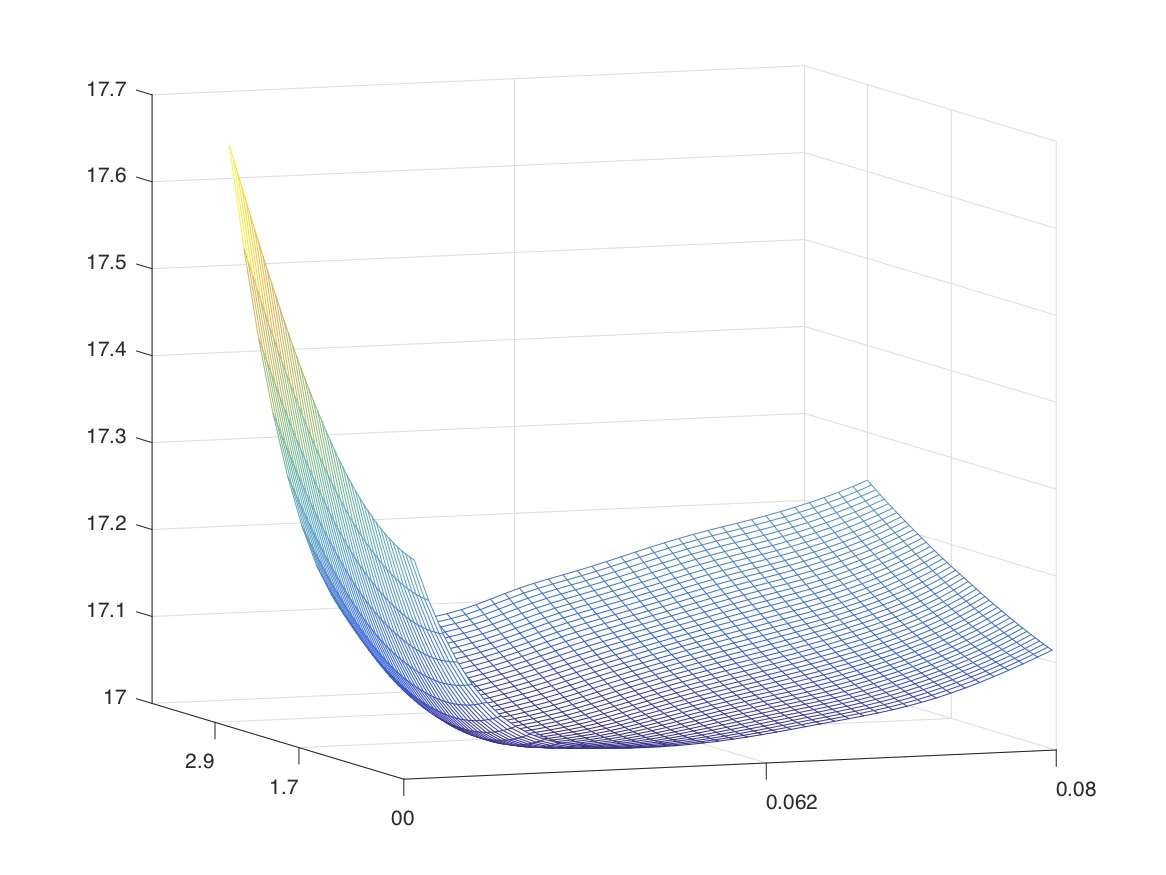}
        \caption{Mesh plot of assessment function}
        \label{fig:step_conter_start}
\end{subfigure}
\begin{subfigure}{.495\textwidth}
  \centering
        \includegraphics[width=1.0\linewidth]{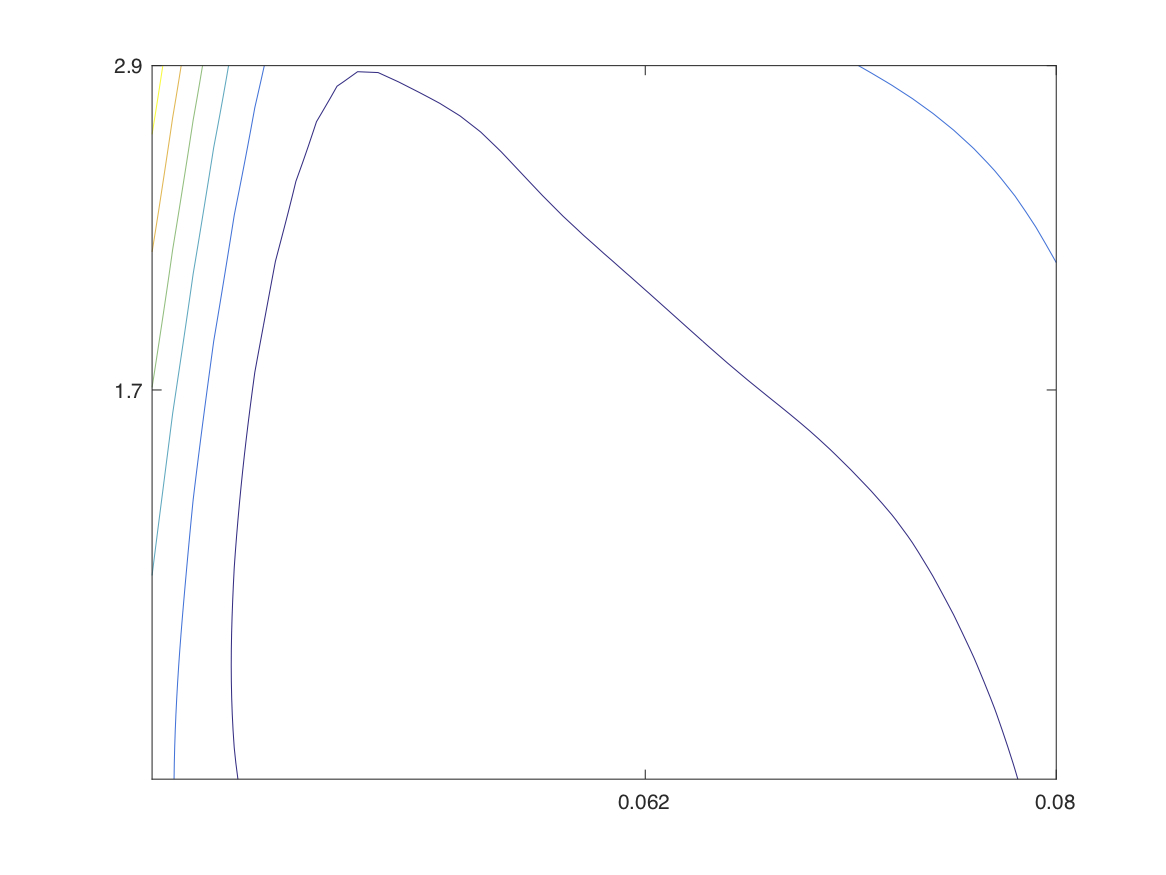}
        \caption{Contour line of assessment function}
        \label{fig:big_side_noise_low_resu_error}
\end{subfigure}
\caption{We note that the assessment function $\mathcal A(\alpha,r,)$ with training ground $\T_3$ is not convex.}
\label{fig:2nd_counter_exam}
\end{figure}

From Figure \ref{fig:2nd_counter_exam} we see that the assessment function with training ground $\T_3$ is not convex, and hence we could expect that others with training ground $\T_1$ and $\T_2$ would not be as well, since solving scheme $\mathcal T$ is equivalent to finding the global minimizer of assessment function $\mathcal A(\alpha,r,p)$. The non-convexity of the assessment function implies that the training scheme $\mathcal T$ might not have a unique global minimizer, i.e., the set $\mathbb A[\T]$ has more than one element. More importantly, the non-convexity of $\mathcal A(\alpha,r,p)$ prevents us from using the standard gradient descent methods to find a global minimizer, since we may get trapped at a local minimum. Hence, a numerical scheme for solving the training scheme $\mathcal T$ (upper level problem) remains an open question. We refer the reader to the recent work \cite{liu2019ATC} for partial results.

\section*{Acknowledgements}
The work of Pan Liu has been supported by the Centre of Mathematical Imaging and Healthcare and funded by the Grant ``EPSRC Centre for Mathematical and Statistical Analysis of Multimodal Clinical Imaging" with No. EP/N014588/1.
Xin Yang Lu acknowledges the support of NSERC Grant 
{\em ``Regularity of minimizers and pattern formation in geometric minimization problems''},
and of the Startup funding, and Research Development Funding
of Lakehead University.

\bibliographystyle{abbrv}
\bibliography{RLS_VReady}{}

\end{document}